\documentclass[11pt, one side, article]{memoir}

\settrims{0pt}{0pt} % page and stock same size
\settypeblocksize{*}{36.1pc}{*} % {height}{width}{ratio}
\setlrmargins{*}{*}{1} % {spine}{edge}{ratio}
\setulmarginsandblock{.98in}{.98in}{*} % height of typeblock computed
\setheadfoot{\onelineskip}{2\onelineskip} % {headheight}{footskip}
\setheaderspaces{*}{1.5\onelineskip}{*} % {headdrop}{headsep}{ratio}
\checkandfixthelayout

\usepackage{amsthm}
\usepackage{mathtools}

\usepackage[inline]{enumitem}
\usepackage{ifthen}
\usepackage[utf8]{inputenc} %allows non-ascii in bib file
\usepackage{xcolor}

\usepackage[backend=biber, backref=true, maxbibnames = 10, style = alphabetic]{biblatex}
\usepackage[bookmarks=true, colorlinks=true, linkcolor=blue!50!black,
citecolor=orange!50!black, urlcolor=orange!50!black, pdfencoding=unicode]{hyperref}
\usepackage[capitalize]{cleveref}

\usepackage{tikz}

\usepackage{amssymb}
\usepackage{newpxtext}
\usepackage[varg,bigdelims]{newpxmath}
\usepackage{mathrsfs}
\usepackage{dutchcal}
\usepackage{fontawesome}
\usepackage{ebproof}
\usepackage{stmaryrd}
\usepackage{float}

\usepackage{listings}
\lstdefinelanguage{HaskellMin}%
  {otherkeywords={=>},%
   morekeywords={if,then,else,case,class,data,deriving,hiding,in,infix,infixl,infixr,import,instance,let,module,newtype,of,qualified,type,where,do,as,family,default,forall,foreign,proc,rec},%
   sensitive,
   morecomment=[l]--,%
   morecomment=[n]{\{-}{-\}},%
   morestring=[b]"%
 }[keywords,comments,strings]

  \crefformat{enumi}{\card#2#1#3}
  \crefalias{chapter}{section}

  \hypersetup{final}

  \setlist{nosep}
  \setlistdepth{6}

  \usetikzlibrary{ 
  	cd,
  	math,
  	decorations.markings,
		decorations.pathreplacing,
  	positioning,
  	arrows.meta,
  	shapes,
  }
  
  \tikzset{
biml/.tip={Glyph[glyph math command=triangleleft, glyph length=.95ex]},
bimr/.tip={Glyph[glyph math command=triangleright, glyph length=.95ex]},
}

\tikzset{
	tick/.style={postaction={
  	decorate,
    decoration={markings, mark=at position 0.5 with
    	{\draw[-] (0,.4ex) -- (0,-.4ex);}}}
  }
} 
\tikzset{
	slash/.style={postaction={
  	decorate,
    decoration={markings, mark=at position 0.5 with
    	{\node[font=\footnotesize] {\rotatebox{90}{$\sim$}};}}}
  }
}

\newcommand{\bifrom}[1][]{
	\begin{tikzcd}[ampersand replacement=\&, cramped]\ar[r, bimr-biml, "{#1}"]\&~\end{tikzcd}  
}
\newcommand{\bifromlong}[2]{
	\begin{tikzcd}[ampersand replacement=\&, column sep=#1, cramped]\ar[r, bimr-biml, "#2"]\&~\end{tikzcd}  
}

\newcommand{\adj}[5][30pt]{%[size] Cat L, Left, Right, Cat R.
\begin{tikzcd}[ampersand replacement=\&, column sep=#1]
  #2\ar[r, shift left=4.5pt, "#3"]
  \ar[r, phantom, "\scriptstyle\Rightarrow"]\&
  #5\ar[l, shift left=4.5pt, "#4"]
\end{tikzcd}
}

\theoremstyle{definition}
\newtheorem{definitionx}{Definition}[chapter]
\newtheorem{examplex}[definitionx]{Example}

\theoremstyle{plain}

\newtheorem{theorem}[definitionx]{Theorem}
\newtheorem{proposition}[definitionx]{Proposition}
\newtheorem{corollary}[definitionx]{Corollary}
\newtheorem{lemma}[definitionx]{Lemma}

\newtheorem*{theorem*}{Theorem}
\newtheorem*{proposition*}{Proposition}
\newtheorem*{corollary*}{Corollary}
\newtheorem*{lemma*}{Lemma}
\newtheorem*{warning*}{Warning}

\newenvironment{example}
  {\pushQED{\qed}\examplex}
  {\popQED\endexamplex}

  \newenvironment{definition}
  {\pushQED{\qed}\definitionx}
  {\popQED\enddefinitionx}

\DeclareSymbolFont{stmry}{U}{stmry}{m}{n}
\DeclareMathSymbol\fatsemi\mathop{stmry}{"23}

\DeclareFontFamily{U}{mathx}{\hyphenchar\font45}
\DeclareFontShape{U}{mathx}{m}{n}{
      <5> <6> <7> <8> <9> <10>
      <10.95> <12> <14.4> <17.28> <20.74> <24.88>
      mathx10
      }{}
\DeclareSymbolFont{mathx}{U}{mathx}{m}{n}
\DeclareFontSubstitution{U}{mathx}{m}{n}
\DeclareMathAccent{\widecheck}{0}{mathx}{"71}

\DeclareMathOperator{\Hom}{Hom}

\DeclareMathOperator*{\colim}{colim}

\DeclareMathOperator{\ob}{Ob}

\newcommand{\cat}[1]{\mathcal{#1}}%a generic category
\newcommand{\Cat}[1]{\mathbf{#1}}%a named category
\newcommand{\id}{\mathrm{id}}
\newcommand{\then}{\mathbin{\fatsemi}}

\newcommand{\tto}{\rightrightarrows}
\newcommand{\To}[2][]{\xrightarrow[#1]{#2}}

\newcommand{\Tto}[3][13pt]{\begin{tikzcd}[sep=#1, cramped, ampersand replacement=\&, text height=1ex, text depth=.3ex]\ar[r, shift left=2pt, "#2"]\ar[r, shift right=2pt, "#3"']\&{}\end{tikzcd}}

\newcommand{\from}{\leftarrow}

\newcommand{\From}[1]{\xleftarrow{#1}}

\newcommand{\card}{\,^{\#}}

\newcommand{\ldag}{^{\rotatebox{180}{$\dagger$}}}

\newcommand{\op}{^\tn{op}}

\newcommand{\tn}[1]{\textnormal{#1}}

\newcommand{\wt}[1]{\widetilde{#1}}

\newcommand{\nn}{\mathbb{N}}

\newcommand{\rr}{\mathbb{R}}

\newcommand{\smset}{\Cat{Set}}

\newcommand{\smcat}{\Cat{Cat}}

\newcommand{\catsharp}{\Cat{Cat}^{\sharp}}

\newcommand{\ccatsharp}{\mathbb{C}\Cat{at}^{\sharp}}

\newcommand{\sspan}{\mathbb{S}\Cat{pan}}

\newcommand{\set}{\tn{-}\Cat{Set}}

\newcommand{\coh}[1]{^{(#1)}}

\newcommand{\yon}{{\mathcal{y}}}
\newcommand{\poly}{\Cat{Poly}}

\newcommand{\cart}{\Cat{cart}}

\newcommand{\ppoly}{\mathbb{P}\Cat{oly}}

\newcommand{\tri}{\mathbin{\triangleleft}}
\newcommand{\ot}[2]{\,\mbox{${}_{#1}\otimes_{#2}$}\,}
\newcommand{\ih}[4]{{}_{#1}[#2,#3]_{#4}}

\newcommand{\qqand}{\qquad\text{and}\qquad}

\newcommand{\coto}{\nrightarrow}

\newcommand{\coalg}{\tn{-}\Cat{Coalg}}

\newcommand{\org}{{\mathbb{O}\Cat{rg}}}

\newcommand{\E}{{\cat{E}}}

\newcommand{\cofree}[1]{\mathfrak{c}_{#1}}
\newcommand{\eff}{\mathbb{E}\Cat{ff}}
\newcommand{\effel}{\mathbb{E}\Cat{ff}^{\,\tn{el}}}
\newcommand{\effellin}{\mathbb{E}\Cat{ff}^{\,\tn{el}}_{\tn{lin}}}

\renewcommand{\path}{\textit{path}}%Strange behavior when using \mathit.

\newcommand{\edge}{\mathrm{E}}
\newcommand{\vertex}{\mathrm{V}}

\newcommand{\biglens}[2]{
     \begin{bmatrix}{\vphantom{f_f^f}#2} \\ {\vphantom{f_f^f}#1} \end{bmatrix}
}
\newcommand{\littlelens}[2]{
     \begin{bsmallmatrix}{\vphantom{f}#2} \\ {\vphantom{f}#1} \end{bsmallmatrix}
}
\newcommand{\lens}[2]{
  \relax\if@display
     \biglens{#1}{#2}
  \else
     \littlelens{#1}{#2}
  \fi
}

\newcommand{\comod}{\mathbb{C}\Cat{omod}}

\newcommand{\minus}{\mathbin{\tn{-}}}

\allowdisplaybreaks
\setsecnumdepth{section}
\settocdepth{section}
\settocdepth{chapter}

\begin{document}

\title{All Concepts are $\ccatsharp$}

\author{Owen Lynch \and Brandon T. Shapiro \and David I. Spivak}

\date{\vspace{-.2in}}

\maketitle

\begin{abstract}
We show that the double category $\ccatsharp$ of comonoids in the category of polynomial functors (previously shown by Ahman-Uustalu and Garner to be equivalent to the double category of small categories, cofunctors, and prafunctors) contains several formal settings for basic category theory and has subcategories equivalent to both the double category $\org$ of dynamic rewiring systems and the double category $\ppoly_\E$ of generalized polynomials in a finite limit category $\E$. Also serving as a natural setting for categorical database theory and generalized higher category theory, $\ccatsharp$ at once hosts models of a wide range of concepts from the theory and applications of polynomial functors and category theory.
\end{abstract}

\renewcommand\cftbeforechapterskip{2pt plus 1pt}

\begin{KeepFromToc}
\tableofcontents
\end{KeepFromToc}

\chapter{Introduction}

Mac Lane famously declared that ``The notion of Kan extensions subsumes all the other fundamental concepts of category theory''  referring to the fact that limits, colimits, adjunctions, and the Yoneda lemma can all be defined in terms of Kan extensions, and titled that section ``All concepts are Kan extensions".

In the theory of polynomial functors, particularly as it has been explored by the authors, the main avenues of development have been the generalization from polynomials in the category $\smset$ to polynomials in other categories \cite{kock2012polynomial,weber2015polynomials,shapiro2023structures} and applications to categorical database theory \cite{spivak2012functorial,spivak2021functorial}, open dynamical systems \cite{spivak2021learners,shapiro2022dynamic}, and algebraic higher category theory \cite{weber2007familial,weber2015operads,shapiro2022thesis}. Recent results of Ahman-Uustalu \cite{ahman2016directed,ahman2017taking} and Garner show that comonoids in the monoidal category $\poly$ of polynomial endofunctors on $\smset$, coincide with the usual notion of categories, comonoid homomorphisms correspond to cofunctors, and bicomodules between comonoids correspond to parametric right adjoint functors between their associated copresheaf categories (also called \emph{prafunctors}). In \cite{spivak2021functorial}, the author assembled these components into a double category $\ccatsharp$ and showed it to be a natural setting for categorical database theory. The author and Brown in \cite{brown2023dynamic} use $\ccatsharp$ as a formal semantics for rewriting protocols, and provide a graphical language for a fragment of it. In \cite{shapiro2024polynomial}, the authors describe how algebraic categorical structures can be modeled in $\ccatsharp$ and show that Weber's nerve of any type of algebraic higher category arises from a universal categorical construction in $\ccatsharp$. The goal of the present work is to demonstrate that $\ccatsharp$ in fact subsumes the other fundamental concepts of polynomial functor theory as well, and begin to describe how basic category theory finds a home (or many) in $\ccatsharp$.

While the objects of $\ccatsharp$ are categories and the vertical and horizontal morphisms (cofunctors and prafunctors) are fundamental to the study of their copresheaf categories, functors between the categories themselves are not explicitly present in the data of $\ccatsharp$, which would seem to limit the usefulness of this setting for modeling basic category theory. Several remedies have been proposed, including by upgrading $\ccatsharp$ to include higher dimensional data \cite[Example 5.13]{shapiro2023structures} or finding functors in alternative places in $\ccatsharp$.%
\footnote{See for instance \href{https://youtu.be/aTZyHHAwk2E?t=7364}{Todd Trimble's talk} at the 2021 Workshop on Polynomial Functors.}
We take the latter approach by considering both monads in the bicategory of spans and algebras for a certain monad on the category of graphs as notions of categories whose morphisms are functors. We show that they can be both modeled in $\ccatsharp$ and recovered from regarding categories as objects in $\ccatsharp$. We also show that opposites of categories can be recovered using adjoint and monoidal dualization operations in $\ccatsharp$.

In \cite{shapiro2023structures}, the authors establish the category $\poly_\E$ of polynomials in a finite limit category $\E$ and show that a wide range of structures and results previously known for polynomials in $\smset$ generalize to this setting. Much like $\poly$, $\poly_\E$ is a duoidal category under composition and a generalization of the Dirichlet tensor product, and comonoids in $\poly_\E$ are precisely the categories internal to $\E$ whose source morphism is exponentiable. \cref{polyEembed} shows that $\poly_\E$ has a faithful embedding into $\ccatsharp$, so that in order to study polynomials in any category $\E$ one need only consider structures based on polynomials in $\smset$. 

%In \cite{weber2007familial}, Weber defined a generalization of the Lawvere theory associated to an algebraic structure on sets, the nerve construction for categories, and generalizations of the nerve to various flavors of algebraic higher categories. For a class of monads $m$ including those whose underlying endofunctor is parametric right adjoint, there is a fully faithful ``nerve'' functor from $m$-algebras to presheaves on a category $\Theta_m$. We show in \cref{nervecoalgebra} that for any parametric right adjoint monad $m$ on a copresheaf category considered as a monad in $\ccatsharp$, the category $\Theta_m\op$ is the comonoid recovered by a right coclosure operation for composition in $\ccatsharp$, and the nerve construction is easily recovered using the universal property of the coclosure. This shows that $\Theta_m$ has a universal property relative to the monad $m$, and opens the door for the further study of higher category theory in the language of polynomials.

In \cite{spivak2021learners,shapiro2022dynamic}, the authors explore how coalgebras for polynomial functors and algebraic structures built from such coalgebras provide a wide-reaching language for modeling dynamical systems which respond to external feedback, and construct the double category $\org$ as a convenient setting for the study of such ``open'' dynamics which includes examples from machine learning and economics. Separately in \cite{Lynch:blog}, the author introduces ``effects handlers,'' a mathematical object defined in terms of polynomials which models a way of incorporating side effects into the functional programming paradigm. In \cref{thm.eff}, we show that effects handlers form the horizontal morphisms of a sub-double category of $\ccatsharp$, and in \cref{thm.org} we show that $\org$ is a sub-double category of effects handlers, so that both effects handlers and coalgebras can be reasoned about in the language of $\ccatsharp$.

Of course, not all concepts in category theory are Kan extensions (for example, categories are not Kan extensions), and neither is every concept in category theory found in $\ccatsharp$. However, in both cases the exaggeration is worthwhile because the overwhelming ubiquity and power of the formalism makes it worthy of deep study. The position and function of $\ccatsharp$ within category theory is akin to the position and function of category theory within mathematics. In both cases, having a single unified and concise formalism---one which covers a broad swath of the larger subject and which has controlled notation and terminology, as well as a praxis of useful thought patterns---facilitates practitioners in finding interesting connections between different fields within the larger subject and concisely communicating their findings to others. Moreover, since $\ccatsharp$ is in some sense the language of data migration \cite{schultz2017algebraic}, everything in this paper can be implemented on a computer in a unified way.

\section*{Notation}

The symbol $\sum$ denotes an indexed coproduct, the symbol $+$ denotes binary coproduct, and $0$ denotes an initial object. For a morphisms $f \colon A \to B$ and $C \to B$ in a category, we will sometimes write $f^\ast C$ for the pullback $A \times_B C$.

\section*{Acknowledgments}

This material is based upon work supported by the Air Force Office of Scientific Research under award number FA9550-20-1-0348. We also appreciate the comments of our ACT2023 conference reviewers, in particular Reviewer uG3q.

\chapter{The Double Category $\ccatsharp$}

We begin by recalling the definition of the double category $\ccatsharp$ and the foundational results that make it so broadly applicable.

\section{The category of polynomials}

\begin{definition}
A \emph{polynomial} $p$ consists of a set $p(1)$ along with, for each element $I \in p(1)$, a set $p[I]$. We write
\[
p = \sum_{I \in p(1)} \yon^{p[I]}
\]
for such a polynomial, which is also the form of the associated \emph{polynomial functor} $\smset \to \smset$. A morphism $\phi$ of polynomials $p \to q$ is a natural transformation. It can be cast set-theoretically as consisting of a function $\phi_1 \colon p(1) \to q(1)$ along with, for each $I \in p(1)$, a function $\phi^\#_I \colon p[I] \from q[\phi_1I]$. We write $\poly$ for the category of polynomials.
\end{definition}

Elements of the set $p(1)$ are called \emph{positions} of a polynomial $p$, and for each $I\in p(1)$, elements of $p[I]$ are called \emph{directions} of $p$. The set of all directions of $p$, or the disjoint union of all the sets $p[I]$, is denoted $p_\ast(1)$ and has a canonical function to $p(1)$. If each $p[I]\cong 1$ is singleton, we say that $p$ is \emph{linear}. A morphism $\phi$ is called \emph{cartesian} if each $\phi^\#_I$ is a bijection, and \emph{vertical} if $\phi_1$ is a bijection.

\begin{definition}[{\cite[Proposition 2.1.7]{spivak2021functorial}}]
We denote by $\yon$ the polynomial with a single position and a single direction. For polynomials $p,q$, their \emph{composition} is the polynomial
\[
p \tri q \coloneq \sum_{\substack{I \in p(1) \\ J \colon p[I] \to q(1)}} \yon^{\sum\limits_{i \in p[I]} q[Ji]}.\qedhere 
\]
\end{definition}

There is a monoidal structure on the category $\poly$ given by $(\yon,\tri)$, and there are three additional monoidal structures given by
\begin{itemize}
	\item $(0,+)$, where $0$ is the polynomial with no positions, $(p + q)(1) \coloneq p(1) + q(1)$, $(p + q)[I] \coloneq p[I]$ for $I \in p(1)$, and $(p + q)[J] \coloneq q[J]$ for $J \in q(1)$;
	\item $(1,\times)$, where $1$ is the polynomial with one position and no directions, $(p \times q)(1) \coloneq p(1) \times q(1)$, and $(p \times q)[I,J] \coloneq p[I] + q[J]$; and
	\item $(\yon,\otimes)$, where $(p \otimes q)(1) \coloneq p(1) \times q(1)$ and $(p \otimes q)[I,J] \coloneq p[I] \times q[J]$.
\end{itemize}
%Moreover, $0$ is initial in $\poly$, $1$ is terminal, $+$ is the coproduct, $\times$ is the product, and $(\yon,\otimes,\tri)$ forms a normal duoidal category structure on $\poly$. %cite
%All of the latter three monoidal structures on $\poly$ are symmetric. %Furthermore, $\poly$ is easily checked %cite anyway
%to have all limits and colimits, where limits of polynomials are given by limits on the sets of positions and the corresponding colimits on the sets of directions, and dually for colimits of polynomials.

\section{Comonoids and bicomodules}

\begin{definition}
A \emph{comonoid} in $\poly$ is a polynomial $c$ equipped with morphisms $\epsilon \colon c \to \yon$ (the counit) and $\delta \colon c \to c \tri c$ (the comultiplication) satisfying unit and associativity equations. A comonoid homomorphism is a morphism of polynomials $c \to c'$ which commutes with counits and comultiplications.
\end{definition}

\begin{definition}
For comonoids $c,d$ in $\poly$, a $(c,d)$-bicomodule is a polynomial $p$, called the \emph{carrier}, equipped with morphisms
\[
c \tri p \From{\lambda} p \To{\rho} p \tri d
\]
which commute with each other as well as the counits and comultiplications of $c$ and $d$, in the sense of \cite[Definition 2.2.11]{spivak2021functorial}. We will often denote a $(c,d)$-bicomodule $p$ as $c \bifrom[p] d$.
\end{definition}

In \cite[Corollary 2.2.10]{spivak2021functorial}, the author established using a theorem of Shulman \cite[Theorem 11.5]{shulman2008framed} that there is a double category $\comod(\poly)$ (in fact an equipment) whose objects are comonoids and horizontal morphisms are bicomodules.

\begin{definition}\label{defcatsharp}
$\ccatsharp$ is the pseudo-double category $\comod(\poly)$ which has
\begin{itemize}
	\item as objects, the comonoids in $\poly$;
	\item as vertical morphisms, the comonoid homomorphisms;
	\item as horizontal morphisms from $c$ to $d$, the $(c,d)$-bicomodules;
	\item as squares between homomorphisms $\phi,\psi$ and bicomodules $p,p'$, the morphisms of polynomials $\gamma \colon p \to p'$ such that the diagram in \eqref{eqn.square} commutes;
	\begin{equation}\label{eqn.square}
	\begin{tikzcd}
	c \tri p \dar[swap]{\phi \tri \gamma} & p \lar \rar \dar{\gamma} & p \tri d \dar{\gamma \tri \psi} \\
	c' \tri p' & p' \lar \rar & p' \tri d'
	\end{tikzcd}
	\end{equation}
	\item as horizontal identities, the comultiplication bicomodules $c \tri c \From{\delta} c \To{\delta} c \tri c$; and
	\item as composition of horizontal morphisms $c \bifrom[p] d \bifrom[q] e$, the bicomodule $p \tri_d q$ on the top row of \eqref{eqn.biccomp},
	\begin{equation}\label{eqn.biccomp}
	\begin{tikzcd}
	c \tri (p \tri_d q) \dar & p \tri_d q \dar \rar \lar & (p \tri_d q) \tri e \dar \\
	c \tri p \tri q \dar[shift right=2] \dar[shift left=2] & p \tri q \dar[shift right=2] \dar[shift left=2] \rar \lar & p \tri q \tri e \dar[shift right=2] \dar[shift left=2] \\
	c \tri p \tri d \tri q & p \tri d \tri q \rar \lar & p \tri d \tri q \tri e
	\end{tikzcd}
	\end{equation}
	where each object in the top row of \eqref{eqn.biccomp} is computed as the equalizer of the column below it, using the fact that the functors $c \tri -$ and $- \tri e$ preserve connected limits, and the maps between them are induced by the underlying transformations between equalizer diagrams. This also shows how to horizontally compose squares, as a pair of adjacent squares provides the data of a transformation of equalizer diagrams which induces a map between the composite bicomodules.\qedhere
\end{itemize}
\end{definition}

\section{Categories, cofunctors, and prafunctors}\label{sec.prafunctors}

%may need to rewrite this section

The motivation for studying $\ccatsharp$ comes from recent results of Ahman--Uustalu \cite{ahman2014when} and Garner%
\footnote{We refer to \href{https://www.youtube.com/watch?v=tW6HYnqn6eI}{Garner's HoTTEST video}, where the proof was sketched; see also \cite{spivak2021functorial}.}
 that, respectively, comonoids in $\poly$ are precisely categories and that bicomodules between them are precisely parametric right adjoint functors (sometimes shortened to \emph{prafunctors}) between their copresheaf categories. This makes $\ccatsharp$ a natural setting for categorical database theory \cite{spivak2012functorial,spivak2021functorial}, where database schemas are categories, instances are copresheaves, and queries (along with more general data migration operations) are prafunctors.

\begin{definition}\label{comonoidcategory}
For a polynomial comonoid $(c,\epsilon,\delta)$, its corresponding (small) category has
\begin{itemize}
	\item as objects, elements of the set $c(1)$;
	\item as morphisms out of an object $C \in c(1)$, the set $c[C]$;
	\item as codomain assignment for morphisms out of $C$, the function $\delta_1(c) \colon c[C] \to c(1)$;
	\item as composition of morphisms out of $C$, the function $\delta^\sharp_C \colon c[C] \times_{c(1)} c_\ast(1) \to c[C]$; and
	\item as the identity morphism at $C$, the function $\epsilon^\#_C \colon 1 \to c[C]$.\qedhere
\end{itemize}
\end{definition}

To go the other way, suppose given a small category $\cat{C}$. For any object $C\in\ob(\cat{C})$, let $\cat{C}[C]\coloneqq\sum_{C'\in\ob(\cat{C})}\cat{C}(C,C')$ denote the set of all morphisms emanating from $C$. Then the polynomial comonad for $\cat{C}$ is carried by the polynomial $c\coloneqq\sum_{C\in\ob(\cat{C})}\yon^{\cat{C}[C]}$. The counit map $\epsilon\colon c\to\yon$ consists of a choice of morphism out of each object, which we take to be the identity. We leave the unpacking of the comultiplication map $\delta\colon c\to c\tri c$---which handles codomains and composition---to the reader; see \href{https://www.youtube.com/watch?v=2mWnrgPIrlA}{this video} for an elementary unpacking.

So comonoids in $\poly$ are (small)%
\footnote{From now on, we may refer to comonoids in $\poly$ simply as categories, rather than emphasizing their smallness.}
categories. Comonoid homomorphisms, however, correspond not to functors but to cofunctors.

\begin{definition}[{\cite[Definition 2.2.2]{spivak2021functorial}}]
For categories $c$ and $d$ (regarded as polynomial comonoids), a cofunctor $c \to d$ is a function $\phi_1 \colon c(1) \to d(1)$ along with, for each $C \in c(1)$, a function $d[\phi_1C] \to c[C]$ which preserves identities, codomains, and composites.
\end{definition}

For $c$ a category, we write $c\set$ for the category of copresheaves on $c$, meaning functors $c \to \smset$. For $X$ a $c$-copresheaf and $C \in c(1)$ an object, we write $X_C$ for $X(C)\in\smset$.

\begin{definition}\label{defprafunctor}
A parametric right adjoint functor $F \colon d\set \to c\set$ is a functor with the following form, for any $d$-copresheaf $X$ and object $C \in c(1)$,
\[
F(X)_C = \sum_{I \in p_C(1)} \Hom_{d\set}(p[I],X)
\]
where $p_{(-)}(1)$ is a functor $c \to \smset$ (which we will denote by simply $p(1)$), $p_C(1)$ is its evaluation at $C$, and $p[-]$ is a functor $\left(\int p(1)\right)\op \to d\set$ from the dual of the category of elements of $p(1)$.

When $p$ is a $(c,d)$-bicomodule and $C \in c(1)$, we have $p(1)\cong\sum_{C\in c(1)}p_C(1)$, and we recover $p_C(1)$ as the preimage of $C$ under the function $p(1) \To{\lambda(1)} (c \tri p)(1) \To{c\tri\,!} c(1)$. Moreover, for $I \in p_C(1)$ and $D \in d(1)$, the set $p[I]_D$ is the preimage of $D$ under the function $p[I] \to d(1)$ given by the element $1\To{I}p(1) \To{\rho(1)} (p \tri d)(1)$.
\end{definition}

Based on this interpretation, we will often denote a $(c,d)$-bicomodule $p$ as
\[
\sum_{C \in c(1)} \sum_{I \in p_C(1)} \yon^{p[I]}
\]
where $p[I]$ is presumed to have the structure of a $d$-copresheaf.

\begin{example}
For any set $A$, the linear polynomial $A\yon$ has a unique comonoid structure; it corresponds to the discrete category on $A$. Cofunctors $A\yon\to B\yon$ are functions $A\to B$.
\end{example}

\begin{example}
For $c$ any category, a $(c,0)$-bicomodule $p$ is a parametric right adjoint functor from $0\set$, the terminal category, to $c\set$. The particular copresheaf on $c$ this functor picks out is $p(1)$, whose elements are positions of $p$ and whose $c$-copresheaf structure is determined by the map $c \tri p \From{\lambda} p$. As there is also a map $p \to p \tri 0$ which preserves the positions of $p$, and forces the polynomial $p$ to have an empty set of directions. The category of $(c,0)$-bicomodules and maps between them as in \eqref{eqn.square} with $\phi,\psi$ identities is equivalent to the category $c\set$.

The composition of a $(c,d)$-bicomodule and a $(d,0)$-bicomodule is precisely the $c$-copresheaf given by applying the prafunctor $d\set \to c\set$ to a $d$-copresheaf.
\end{example}

\begin{example}
More generally, a parametric right adjoint functor $d\set \to c\set$ is a right adjoint precisely when it arises from a profunctor from $c$ to $d$: a copresheaf on $c\op \otimes d$ induces a functor $c\op \to d\set$ whose corresponding singular functor is a right adjoint $d\set \to c\set$. However, unlike when $d$ is discrete, the left adjoint of this prafunctor will not generally be a prafunctor itself.
\end{example}

We now describe how the identities and composition of bicomodules given in \cref{defcatsharp} behave under the correspondence with prafunctors from \cref{defprafunctor}. For a $(c,d)$-bicomodule $p$ and a $(d,e)$-bicomodule $q$ of the forms
\[
p = \sum_{C \in c(1)}\sum_{I \in p_C(1)} \yon^{p[I]} \qqand q = \sum_{D \in d(1)}\sum_{J \in q_D(1)} \yon^{q[J]},
\]
the equalizer of $p \tri q \tto p \tri d \tri q$ has as positions the subset of functions $p[I] \to q(1)$ which are morphisms between the associated $d$-copresheaf structures; this is because the two maps to $p \tri d \tri q$ each append such a map with the data of either the restrictions of elements of $p[I]$ under maps in $d$ or those of elements in $q(1)$, which in the equalizer must agree. The directions for a position given by $f \colon p[I] \to q(1)$ is the coequalizer of the disjoint union 
\[
(p \tri q)[I,f] = \sum_{i \in p[I]} q[f(i)]
\]
under the maps respectively sending $i$ to its restrictions along maps in $d$ within the $d$-copresheaf $p[I]$ and mapping $q[f(i)]$ to the arities of the restriction of $f(i)$ along maps in $d$ according to the left $d$-module structure of $q$ on directions. These identifications turn the disjoint union $(p \tri q)[I,f]$ into the corresponding colimit 
\[
(p \tri_d q)[I,f] = \colim_{i \in p[I]} q[f(i)]
\]
indexed by the category of elements of $p[I]$ as a $d$-copresheaf. It is easily checked (as stated in \cite[Proposition 1.8]{shapiro2022thesis} and a consequence of the proofs of \cite[Propositions 3.11,3.12]{garner2018shapely}) that these positions and directions agree with those of the composite of the corresponding parametric right adjoint functors.

The identity bicomodule $c \bifrom[c] c$ has the form $\sum\limits_{C \in c(1)} \yon^{c[C]}$, so it has a single operation for each object of $c$ with arity the corepresentable copresheaf $c[C]$.

\section{Right coclosure and left Kan extension}

We now recall the \emph{right coclosure} or \emph{left Kan extension}.

\begin{definition}[{\cite[Proposition 2.4.6]{spivak2021functorial}}]\label{coclosure}
For a $(d,e)$-bicomodule $q$, the functor $- \tri_d q$ from $(c,d)$-bicomodules to $(c,e)$-bicomodules has a left adjoint $\lens{-}{q}$. For a $(c,e)$-bicomodule $p$ its carrier is defined to be
\begin{equation}\label{eqn.coclosure}
\lens{p}{q} \coloneq \sum_{C \in c(1)}\sum_{I \in p_C(1)} \yon^{q \tri_e p[I]},
\end{equation}
where $p[I]$ is regarded as an $(e,0)$-bicomodule.
\end{definition}

We note the unit and counit of this adjunction for convenience:
\begin{equation}\label{eqn.coclosure_unit_counit}
  p\to\lens{p}{q}\tri q
  \qqand
	\lens{r \tri q}{q}\to r
\end{equation}
The former illustrates how the right coclosure from \eqref{eqn.coclosure} corresponds to the left Kan extension, equivalently in $\ccatsharp$ and the bicategory of copresheaf categories and familial functors.
\[
\begin{tikzcd}
	c\ar[r,bimr-biml, "p", ""' name=p]&
	e\ar[d,biml-bimr, "q"]\\&
	d\ar[ul, bend left, biml-bimr, pos=.8, "\lens{p}{q}", ""' name=cocl]
	\ar[from=p, to=p|-cocl.south, shorten >=-8pt, Rightarrow]
\end{tikzcd}
\hspace{1in}
\begin{tikzcd}
	e\set\ar[d, "q\tri-"']\ar[r, "p\tri-", ""' name=p]&
	c\set\\
	d\set\ar[ur, bend right, "\tn{Lan}_pq"', "" name=cocl]
	\ar[from=p, to=p|-cocl.south, shorten >=0pt, Rightarrow]
\end{tikzcd}
\]

\begin{lemma}\label{lemma.selection}
Given a polynomial $p$ and a polynomial comonoid $c$, the right coclosure $\lens{p\tri c}{p}$ is also a comonoid.
\end{lemma}
\begin{proof}
We need to produce a comonoid structure
\[
\lens{p\tri c}{p}\to\yon
\qqand
\lens{p\tri c}{p}\to\lens{p\tri c}{p}\tri \lens{p\tri c}{p}
\]
In both cases we use the universal properties from \eqref{eqn.coclosure_unit_counit}:
\begin{gather*}
  \lens{p\tri c}{p}\to
 	\lens{p}{p}\to\yon
\\
  \lens{p\tri c}{p}\to
  \lens{p\tri c\tri c}{p}\to
	\lens{\lens{p\tri c}{p}\tri p\tri c}{p}\to
	\lens{\lens{p\tri c}{p}\tri\lens{p\tri c}{p}}{p}\to
	\lens{p\tri c}{p}\tri\lens{p\tri c}{p}
\end{gather*}
It is routine to show that this is associative and unital.
\end{proof}

For a more detailed description of this category, see \cite{Spivak:blogselection}.

\chapter{Basic Category Theory in $\ccatsharp$}

While cofunctors and prafunctors are interesting and useful branches of category theory, they are not the stuff of a category theorist's typical toolbox. However, traditional features of category theory can also be recovered in $\ccatsharp$ by various means which we now discuss.

\section{Products and coproducts}

Both monoidal products $\otimes$ and $+$ have a \emph{duoidal} relationship with composition $\tri$, meaning there are natural morphisms
\begin{equation}\label{eqn.duoidal}
(p \tri q) \otimes (r \tri s) \to (p \otimes r) \tri (q \otimes s) \qqand (p \tri q) + (r \tri s) \to (p + r) \tri (q + s).
\end{equation}
As a general consequence of duoidality, comonoids in $\poly$ are closed under $+$ and $\otimes$.

\begin{theorem}[{\cite[Proposition 2.6.2]{spivak2021functorial}}]
For categories $c,d$ regarded as polynomial comonoids, $c + d$ corresponds to the usual coproduct and $c \otimes d$ to the usual product of $c$ and $d$ as categories. Similarly, $0$ corresponds to the empty category and $\yon$ to the terminal category.
\end{theorem}

\begin{example}
For categories $c,d$, there is a bicomodule $c \otimes d \bifrom[c \times d] c + d$ where the set $(c \times d)(1) = c(1) \times d(1)$ forms the elements of the terminal copresheaf on $c \otimes d$ and each direction set $(c \times d)[C,D] = c[C] + d[D]$ forms the elements of the copresheaf $(c[C],d[D])$ in $(c+d)\set \simeq c\set \times d\set$. The corresponding prafunctor sends the pair $(X,Y)$ of copresheaves $X$ on $c$ and $Y$ on $d$ to the copresheaf $X \boxtimes Y$ on $c \otimes d$ with
\begin{align*}
	(X \boxtimes Y)_{C,D} &= 
  \Hom_{(c+d)\set}\left((c[C],d[D]),(X,Y)\right)\\&\cong 
  \Hom_{c\set}(c[C],X) \times \Hom_{d\set}(d[D],Y) \cong X_C \times Y_D.
\end{align*}
because $c[C]$ and $d[D]$ correspond to representable copresheaves. The prafunctor we have thus described is sometimes called the \emph{external product on copresheaves}.
\end{example}

\section{Three homes for categories}

We now show how categories live in $\ccatsharp$ in at least three different ways, and how to mediate between them. Categories are, simultaneously:
\begin{itemize}
	\item comonoids in $\poly$, and hence objects in $\ccatsharp$ (\cref{comonoidcategory});
	\item algebras for the parametric right adjoint monad $\path$ on graphs (\cref{pathmonad}) \cite[Section II.7]{MacLane:1998a}; and
		\item monads in the double subcategory of $\ccatsharp$ consisting of linear comonoids and linear bicomodules (spans) \cite[5.4.3]{benabou1967introduction}.
\end{itemize}

\begin{definition}\label{def.graph_g}
We denote by $g$ the category $\edge \Tto{s}{t} \vertex$ whose copresheaves are precisely graphs, and by $\vec n$ the graphs with vertices $0,...,n$ and edges $i \minus 1 \to i$ for all $1 \le i \le n$.
\end{definition}

\begin{definition}\label{pathmonad}
The bicomodule $g \bifrom[\path] g$ has carrier given by $\{\vertex\}\yon + \{\edge\}\sum\limits_{n \in \nn} \yon^{\vec n}$, where the labels $\edge,\vertex$ indicate how the left coaction is defined on positions.
\end{definition}

This is a monad in $\ccatsharp$ whose corresponding prafunctor is the free category monad on graphs: it is the identity on vertices and adds in formal associative composites for paths of edges with any length $n$, which are precisely the maps into a graph from $\vec n$ \cite[Example C.3.3]{Leinster:2004a}. A category is then precisely a graph $X$, which can be modeled as a $(g,0)$-bicomodule, equipped with a left module structure of the form $\path \tri_g X \to X$. 

Given a category $c$, there is a bicomodule $g \bifromlong{50pt}{\{\vertex\}c + \{\edge\}c_\ast} c$, where $c_\ast \coloneq \sum\limits_{C \in c(1)} c[C]\yon^{c[C]}$. The left $g$-comodule structure arises from the cartesian source and target morphisms $c_\ast \to c$, while the right $c$-comodule structure is given by the comultiplication $c \to c \tri c$ and its composition with the source morphism $c_\ast \to c$.

The corresponding prafunctor $c\set \to g\set$ sends a copresheaf $X$ on $c$ to the graph for which a vertex is an element of $X$ and an edge is a pair of a morphism in $c$ and an element of $X$ over its source object. This is precisely the underlying graph of the category of elements of $X$, and as such $c+c_\ast$ has a left $\path$-module structure $\path \tri_g (c+c_\ast) \to c+c_\ast$ which induces by precomposition a left $\path$-module on $(c + c_\ast) \tri_c X$ for any copresheaf $X$: this $\path$-algebra is precisely $X$'s category of elements. Applying this to the terminal copresheaf $c \bifrom[c(1)] 0$ recovers the category $c$ itself as a $\path$-algebra.%is there a theorem in this?

Furthermore, for any functor $f$ from $c$ to $d$, there is a bicomodule $c \bifrom[\Delta_f] d$, where $\Delta_f\coloneqq \sum\limits_{C \in c(1)} \yon^{d[f(C)]}$. It comes equipped with a canonical morphism $(c + c_\ast) \tri_c \Delta_f \to d + d_\ast$ of $(c,d)$-bicomodules which commutes with the $\path$-module structures of $c+c_\ast$ and $d+d_\ast$. As $\Delta_f \tri_d d(1) \cong c(1)$ as $(c,0)$-bicomodules, we have constructed in $\ccatsharp$ the morphism of $\path$-algebras corresponding to the functor $f$.

%maybe include grothendieck construction as factorization of a prafunctor into a profunctor and a cartesian cofunctor to recover c from a path algebra

%\cref{comonoidcategory} shows how to see categories as objects of $\ccatsharp$, but not in a way that recovers functors between them. \cref{catsigmafree} shows how to see categories as linear monads in $\ccatsharp$ between discrete objects and functors as maps between them, but in a manner quite distinct from viewing categories as polynomial comonoids. Categories have also been characterized %cite
%as algebras for the $\path$ monad on graphs, which adds in formal composite edges for all finite length paths in a graph. We now show how $\path$ lives in $\ccatsharp$

%While \cref{catssigmafree} shows one way in which how the category of categories and functors can be recovered from $\ccatsharp$, it does not make use of the fact that categories are the objects of $\ccatsharp$. We now show how to connect these two notions of category and identify constructions of profunctors and duals of categories in $\ccatsharp$.

We now describe how each object in $\ccatsharp$ also gives rise to a monad among spans, using the fact that for discrete categories $A\yon,B\yon$ an $(A\yon,B\yon)$-bicomodule $p$ can be summarized by a diagram $B \From{g} p_\ast(1) \to p(1) \To{f} A$ of sets and functions. The left coaction $A\yon \tri p \from p$ is cartesian and given on positions by $\langle f,\id \rangle \colon p(1) \to A \times p(1)$, and the right coaction $p \to p \tri B\yon$ is also cartesian and on positions sends $I \in p(1)$ to $(I, g_{p[I]} \colon p[I] \to B)$.

For any category $c$, there is a bicomodule $c(1)\yon \bifrom[c] c(1)\yon$ given by the diagram
\begin{equation}\label{eqn.pispan}
c(1) \From{t} c_\ast(1) \To{s} c(1) \To{\id} c(1)
\end{equation}
where the left and middle functions are respectively the target and source functions from the set $c_\ast(1)$ of morphisms in $c$ to the set of objects $c(1)$.

By \cite[Proposition 2.5.4]{spivak2021functorial}, a bicomodule between discrete categories whose rightmost function is an identity \eqref{eqn.pispan} is always a right adjoint in $\ccatsharp$, whose left adjoint is the bicomodule given by the diagram
\[
c(1) \From{s} c_\ast(1) \To{\id} c_\ast(1) \To{t} c(1).
\] 
By \cite[Proposition 2.5.6]{spivak2021functorial}, as $c(1) \bifrom[c] c(1)$ is a comonad in $\ccatsharp$ its left adjoint $c(1) \bifrom[c\ldag] c(1)$ is a linear monad in $\ccatsharp$, i.e.\ a monad in $\mathbb{S}\Cat{pan}$, i.e.\ a category. This gives a third home for the category $c$. As desired, for categories $c,d$ a functor between them is a monad map between their corresponding left adjoint spans, so this provides another encoding of functors in $\ccatsharp$.

\section{Opposites}\label{sec.op}

Representing categories as spans allows for a construction of dual categories using only universal constructions in $\ccatsharp$. In \cite[Proposition 2.7.3]{spivak2021functorial}, the author defines a \emph{closure} for the category of $(c,d)$-bicomodules. When $c = A\yon$ and $d = B\yon$, this has the form
\[
\ih{A\yon}{p}{q}{B\yon} \coloneq \sum_{\substack{a \in A \\ \phi \colon p_a \to q_a}} \yon^{\sum\limits_{I \in p_a(1)} q[\phi_1(I)]}
\]
where the maps $p_a \to q_a$ are morphisms of $(\yon,B\yon)$-bicomodules. We can then define a \emph{dualizing} operation for $(A\yon,B\yon)$-bicomodules by setting 
\[
p^\vee \coloneq \ih{A\yon}{p}{AB\yon}{B\yon} = \sum_{a \in A} \Hom(p_a, B\yon) \yon^{p_a(1)}.
\]
In particular, this dual interpolates between left-adjoint bicomodules of the form $B \from C = C \to A$ and right-adjoint bicomodules of the form $B \from C \to A = A$.

This allows spans from $A$ to $A$, modeled as left-adjoint $(A\yon,A\yon)$-bicomodules, to be reversed using only adjunctions and duals: given a left adjoint $p$ represented by $A \From{f} C = C \To{g} A$, its adjoint $p\ldag$ is represented by $A \From{g} C \To{f} A = A$ and its dual $p^\vee$ by $A \From{f} C \To{g} A = A$, so both $(p\ldag)^\vee$ and $(p^\vee)\ldag$ are represented by $A \From{f} C = C \To{g} A$, the reverse of $p$.

\begin{theorem}
For $c$ a category regarded as a $(c(1),c(1))$-bicomodule, its opposite category $c\op$ is given by the $(c(1),c(1))$-bicomodule $(c\ldag)^\vee \cong (c^\vee)\ldag$.
\end{theorem}

\chapter{Generalized Polynomials in $\ccatsharp$}

Much of the development of the theory of polynomials (for instance \cite{kock2012polynomial,weber2015polynomials,shapiro2023structures}) is focused on generalizing the basic aspects of the theory to categories other than $\smset$. We show that, in fact, these categories of polynomials embed fully faithfully into categories of bicomodules, so that the constructions in these contexts are merely specializations of the analogous constructions for bicomodules in $\ccatsharp$.

\section{Polynomials in a category $\E$}

Throughout this section, let $\E$ be a category with pullbacks. Polynomials in $\E$ will generalize the definition of polynomials as morphisms $p_\ast(1) \to p(1)$ in $\smset$.

\begin{definition}
A \emph{polynomial} in $\E$ is an exponentiable morphism $p \colon P_\ast \to P$ in $\E$, and a \emph{morphism of polynomials} $p \to q$ in $\E$ consists of a morphism $P \to Q$ and a morphism $P_\ast \from P \times_Q Q_\ast$. We denote by $\poly_\E$ the category of polynomials in $\E$.

A \emph{typed polynomial} from $D$ to $C$ in $\E$ is a diagram $D \from P_\ast \to P \to C$ such that $P_\ast \to P$ is exponentiable, generalizing the definition of multivariable polynomials in $\smset$ of the form $D \from p_\ast(1) \to p(1) \to C$.
\end{definition}

Here the function $p(1) \to C$ separates the terms of the polynomial into $|C|$-many components (as in a polynomial function $\rr^D \to \rr^C$) while the function $D \from p_\ast(1)$ assigns the variable names from $D$ to the factors of each term in the polynomial.

\begin{definition}[{Based on \cite[Section 3]{kock2012polynomial} and \cite[Section 5.3]{shapiro2023structures}}]\label{defppoly}
The double category $\ppoly_\E$ of typed polynomials in $\E$ has
\begin{itemize}
	\item as objects, objects of $\E$;
	\item as vertical morphisms, morphisms of $\E$;
	\item as horizontal morphisms from $D$ to $C$, typed polynomials from $D$ to $C$;
	\item as squares between morphisms $f,g$ and typed polynomials $P_\ast \to P$ and $P'_\ast \to P'$, isomorphism classes of commuting diagrams as in \eqref{eqn.typedmorphism}, where the isomorphisms are those between choices of pullbacks which commute with the rest of the diagram;
\begin{equation}\label{eqn.typedmorphism}
\begin{tikzcd}
D \ar{dd}[swap]{f} & P_\ast \rar \lar & P \dar[equals] \rar & C \ar{dd}{g} \\
& \bullet \dar \uar{\phi^\ast} \rar \ar[phantom]{dr}[pos=0]{\lrcorner} & P \dar{\phi_1} \\
D' & P'_\ast \lar \rar & P' \rar & C'
\end{tikzcd}
\end{equation}
	\item as horizontal identities, typed polynomials of the form $C = C = C = C$; and
	\item composition of typed polynomials $P_\ast \to P$ and $Q_\ast \to Q$ given by the composite of the top row of morphisms of \eqref{eqn.typedcompose},
\begin{equation}\label{eqn.typedcompose}
\begin{tikzcd}
& \bullet_2 \ar{dd} \ar{rr} \ar[phantom]{ddrr}[pos=0]{\lrcorner} &[25pt] &[-20pt] \bullet_1 \ar{rr} \dar \ar[phantom]{ddrr}[pos=0]{\lrcorner} &[-25pt] & \Pi_p (Q \times_D P_\ast) \ar{dd} &[-25pt] {} \\
& & & Q \times_D P_\ast \ar{dl} \ar{dr} \\
& Q_\ast \rar{q} \ar{dl}[swap]{} & Q \ar{dr}{} & {} & P_\ast \rar{p} \ar{dl}[swap]{} & P \ar{dr}{} \\
E & & & D & & & C
\end{tikzcd}
\end{equation}
where $\Pi_p(Q \times_D P_\ast)$ is defined by the universal property that morphisms into it from an object $A$ correspond to pairs $(f_1 \colon A \to P, f_2 \colon f_1^\ast P_\ast \to Q)$ with $f_2$ commuting over $D$; in other words, the pullback square on the right in \eqref{eqn.typedcompose} is terminal among pullbacks of $p$ whose projection to $P_\ast$ factors through $Q \times_D P_\ast$.
\end{itemize}
As we discuss in the proof of \cref{polyEembed}, as we are constructing a locally fully faithful double functor out of $\ppoly_\E$ there is no need to define horizontal composition of squares as it can be deduced from horizontal composition in $\ccatsharp$. Vertical composition of squares is as given for morphisms of untyped polynomials in \cite[Definition 3.13]{shapiro2023structures}, though similarly this can be deduced from the vertical composition of squares in $\ccatsharp$.
\end{definition}

In particular, when $\E$ has finite limits we see that $\poly_\E$ is a monoidal category \cite[Section 3.2]{shapiro2023structures} as it agrees with the category $\ppoly_\E(1,1)$ with the monoidal structure given by the horizontal identity and composition.

\section{Embedding $\ppoly_\E$ into $\ccatsharp$}

As discussed in \cite[proof of Theorem 3.15]{shapiro2023structures}, the category $\poly_\E$ embeds fully faithfully into $\poly_{a\set}$ for $F \colon a\op \to \E$ any fully faithful dense functor, e.g.\ the identity functor for $a \coloneqq \E\op$. For such an $F$, let $F^\ast\colon\E\to a\set$ be given by $F^\ast(C)(A)\coloneqq\E(F(A),C)$.

\begin{theorem}\label{polyEembed}
For a fully faithful dense functor $F\colon a\op \to \E$, the category $\poly_\E$ embeds fully faithfully into the horizontal category $\ccatsharp(a,a)$. When $\E$ has finite limits, so composition can be defined, this functor is strong monoidal. In particular, for a polynomial $P_\ast \to P$ in $\E$, the corresponding $(a,a)$-bicomodule is given by
\[
\sum_{A \in a(1)} \sum_{x \colon F(A) \to P} \yon^{F^\ast(x^\ast P_\ast)}.
\]
More generally, when $\E$ has pullbacks there is a locally fully faithful double functor from $\ppoly_\E$ to $\ccatsharp$. It sends an object $C$ to the slice category $F/C \cong \int F^\ast(C)$, and a typed polynomial $D \from P_\ast \to P \to C$ to the $(F/C,F/D)$-bicomodule
\[
\sum_{\substack{A \in a(1) \\ F(A) \to C}} 
\sum_{x \in \Hom_{\E/C}(F(A),P)}
%\sum_{\tiny{\begin{tikzcd}[row sep=0,column sep=tiny] x \colon F(A) \ar{rr} \ar[shorten=-2]{dr} & & P \ar[shorten >=-3]{dl} \\ & C \end{tikzcd}}}
\yon^{F^\ast(x^\ast P_\ast)}.\qedhere
\]
\end{theorem}

%\begin{theorem}\label{proofPolyE}
%For a fully faithful dense functor $F \colon a\op \to \E$, there is a locally fully faithful pseudo-double functor from $\ppoly_\E$ to $\ccatsharp$. It sends an object $C$ to the slice category $F/C \cong \int F^\ast(C)$, and a typed polynomial $D \from P_\ast \to P \to C$ to the $(F/C,F/D)$-bicomodule
%\[
%\sum_{\substack{A \in a(1) \\ F(A) \to C}} 
%\sum_{\tiny{\begin{tikzcd}[row sep=tiny,column sep=tiny] x \colon F(A) \ar{rr} \ar{dr} & & P \ar{dl} \\ & C \end{tikzcd}}}
%\yon^{F^\ast(x^\ast P_\ast)}.
%\]
%\end{theorem}

Here ``locally fully faithful'' means that for any fixed square boundary in $\ppoly_\E$, the function from its square fillers to squares with the corresponding boundary in $\ccatsharp$ is a bijection. In particular this implies that the category of typed polynomials in $\E$ from $D$ to $C$ maps fully faithfully to the category of $(F/C,F/D)$-bicomodules. In the case when $\E$ has finite limits, the first statement of \cref{polyEembed} follows from setting $C$ and $D$ to be the terminal object, resulting in a fully faithful strong monoidal functor from $\poly_\E$ to $(a,a)$-bicomodules.

Note that $F^\ast(x^\ast P_\ast)$, as an $a$-copresheaf over $F^\ast(D)$, is equivalently regarded as a copresheaf on $F/D$.

\begin{proof}%%%upgrade as composiiton of functors, beware y's and yon's
Following the approach of \cite[Section 3.2]{shapiro2023structures}, as the assignment $C \mapsto F/C$ is clearly functorial on the vertical categories, it suffices to show that the assignment on horizontal morphisms preserves identities and composition up to coherent isomorphism and that the given assignments are indeed locally fully faithful. The remaining structure and properties of a pseudo-double functor can then be deduced using local fully faithfulness, in the style of \cite[Proposition 3.25]{shapiro2023structures}, as can the composition of squares in $\ppoly$.\footnote{The specific analogue of that proposition would proceed by: 1) defining a tentative pseudo-double category as a pair of categories with the same objects and sets of squares filling boundaries of the appropriate type; 2) defining a tentative pseudo-double functor as an assignment on the categories and squares preserving vertical composition strictly and horizontal composition up to bidirectional squares; and 3) concluding that a tentative pseudo-double category with a locally fully faithful tentative pseudo-double functor to an established pseudo-double category endows the domain with the structure of a pseudo-double category such that the tentative pseudo-double functor is in fact a pseudo-double functor.} 

The identity polynomial $C = C = C = C$ is sent to the $(F/C,F/C)$-bicomodule
\[
\sum_{\substack{A \in a(1) \\ F(A) \to C}} \sum_{x \in \Hom_{\E/C}(F(A),C)} \yon^{F^\ast(x^\ast C)} \quad\cong\quad \sum_{\substack{A \in a(1) \\ F(A) \to C}} \yon^{F^\ast(F(A))} \quad\cong\quad \sum_{\substack{A \in a(1) \\ F(A) \to C}} \yon^{a[A]},
\]
as since $F$ is fully faithful $F^\ast(F(A)) \cong a[A]$. As an $F/C$-copresheaf, thiscopy of $a[A]$ corresponds to the copresheaf corepresented by the map $F(A) \to C$, whose elements are in bijection with the set $a[A]$. This is precisely the form of the identity $(F/C,F/C)$-bicomodule, so our desired double functor preserves horizontal identities.

For typed polynomials $D \from P_\ast \To{p} P \to C$ and $E \from Q_\ast \To{q} Q \to D$, their composite in $\ppoly_\E$ is sent to the $(F/C,F/E)$-bicomodule
\[
\sum_{\substack{A \in a(1) \\ F(A) \to C}} \sum_{x \in \Hom_{\E/C}(F(A),\Pi_p(Q \times_D P_\ast))} \yon^{F^\ast(x^\ast \bullet_2)},
\]
where $\bullet_2$ is defined via pullbacks in \eqref{eqn.typedcompose}, and the composite of the associated bicomodules in $\ccatsharp$ is the $(F/C,F/E)$-bicomodule
\[
\sum_{\substack{A \in a(1) \\ F(A) \to C}} \sum_{\substack{x_1 \in \Hom_{\E/C}(F(A),P) \\ x_2 \in \Hom_{d\set/F^\ast(D)}(F^\ast(x_1^\ast P_\ast),F^\ast(Q))}} \yon^{\colim\limits_{y \colon a[A'] \to F^\ast(x_1^\ast P_\ast)} y^\ast F^\ast(x_2^\ast Q_\ast)}.
\]

By the universal property of $\Pi_p(Q \times_D P_\ast)$, a morphism $x \colon F(A) \to \Pi_p(Q \times_D P_\ast)$ commuting over $C$ corresponds to a morphism $x_1 \colon F(A) \to P$ commuting over $C$ along with a map $\bar x_2 \colon x_1^\ast P_\ast \to Q$ commuting over $D$. As the functor $F^\ast$ is fully faithful, maps of the form $x_2$ and $\bar x_2$ are is bijective correspondence, so these bicomodules agree on positions.

To compute the pullback $x^\ast \bullet_2$ in terms of the maps $x_1,x_2$, consider the extension of \eqref{eqn.typedcompose} given in \eqref{eqn.composepb}.
\begin{equation}\label{eqn.composepb}
\begin{tikzcd}
& x_2^\ast Q_\ast \ar{rr} \dar \ar[phantom]{dr}[pos=0]{\lrcorner} &[25pt] &[-20pt] x_1^\ast P_\ast \ar{rr} \dar \ar[phantom]{dr}[pos=0]{\lrcorner} \ar[bend right=20]{dddl}[swap]{x_2} &[-25pt] & F(A) \dar{x} \ar[bend left=60]{ddd}{x_1} &[-25pt] {} \\
& \bullet_2 \ar{dd} \ar{rr} \ar[phantom]{ddrr}[pos=0]{\lrcorner} & {} & \bullet_1 \ar{rr} \dar \ar[phantom]{ddrr}[pos=0]{\lrcorner} & {} & \Pi_p (Q \times_D P_\ast) \ar{dd} \\
& & & Q \times_D P_\ast \ar{dl} \ar{dr} \\
& Q_\ast \rar{q} \ar{dl}[swap]{} & Q \ar{dr}{} & {} & P_\ast \rar{p} \ar{dl}[swap]{} & P \ar{dr}{} \\
E & & & D & & & C
\end{tikzcd}
\end{equation}
The pullback $x^\ast \bullet_1$ agrees with $x_1^\ast P_\ast$ by the cancellation property of pullbacks, as $x_1$ factors through $x$. Similarly, as $x_2$ factors through the projection $x_1^\ast P_\ast \to \bullet_1$, the pullback of the latter to $\bullet_2$ agrees with $x_2^\ast Q_\ast$. By composition of pullbacks then, we have that $x^\ast \bullet_2 \cong x_2^\ast Q_\ast$, so to show that our desired double functor indeed preserves horizontal composition it suffices to show that
\begin{equation}\label{eqn.directionsiso}
F^\ast(x_2^\ast Q_\ast) \cong \colim_{y \colon a[A'] \to F^\ast(x_1^\ast P_\ast)} y^\ast F^\ast(x_2^\ast Q_\ast).
\end{equation}

To see this, recall the canonical colimit decomposition
\begin{equation}\label{eqn.canonicalcolim}
F^\ast (x_1^\ast P_\ast) \cong \colim_{y \colon a[A'] \to F^\ast(x_1^\ast P_\ast)} a[A']
\end{equation}
of an object in a copresheaf category. As $a\set$ is locally cartesian closed, the pullback functor 
\[
a\set/F^\ast(x_1^\ast P_\ast) \to a\set/F^\ast(x_2^\ast Q_\ast)
\]
is a left adjoint and therefore preserves colimits. In the case of the colimit in \eqref{eqn.canonicalcolim}, this colimit preservation shows that \eqref{eqn.directionsiso} holds, as the left side is the pullback of the identity on $F^\ast(x_1^\ast P_\ast)$ to $F^\ast(x_2^\ast Q_\ast)$ and the right side is the colimit of the pullbacks $y^\ast F^\ast(x_2^\ast Q_\ast)$ of each map $a[A'] \to F^\ast(x_1^\ast P_\ast)$ along the same map. This completes the proof that our desired double functor preserves horizontal composition up to isomorphism.

It then remains to show local fully faithfulness. Consider an arrangement of typed polynomials as in \eqref{eqn.squareframe}.
\begin{equation}\label{eqn.squareframe}
\begin{tikzcd}
D \dar[swap]{f} & P_\ast \rar \lar & P \rar & C \dar{g} \\
D' & P'_\ast \lar \rar & P' \rar & C'
\end{tikzcd}
\end{equation}

A square filling in the associated diagram in $\ccatsharp$ has the form of a polynomial morphism
\[
\sum_{\substack{A \in a(1) \\ z \colon F(A) \to C}} 
\sum_{x \in \Hom_{\E/C}(F(A),P)}
\yon^{F^\ast(x^\ast P_\ast)}
\To{\phi}
\sum_{\substack{A \in a(1) \\ z' \colon F(A) \to C'}} 
\sum_{x' \in \Hom_{\E/C'}(F(A),P')}
\yon^{F^\ast(x'^\ast P'_\ast)}
\]
where $(A,z \colon F(A) \to C)$ is sent to the composite $(A,g \circ z \colon F(A) \to C \to C')$, the maps
\[
\phi_1^{A,z} \colon \Hom_{\E/C}(F(A),P) \to \Hom_{\E/C'}(F(A),P')
\]
are natural in $A$ and $z \colon F(A) \to C$, and the maps of $a$-copresheaves on directions
\[
\phi^\#_x \colon F^\ast(\phi_1^{A,z}(x)^\ast P'_\ast) \to F^\ast(x^\ast P_\ast)
\]  
are natural in $x$ (as an object in the category of elements of $F^\ast(P)$) and commute with $F^\ast(f) \colon F^\ast(D) \to F^\ast(D')$. The maps $\phi_1^{A,z}$ assemble into a map $F^\ast(P) \to F^\ast(P')$ commuting with $F^\ast(g) \colon F^\ast(C) \to F^\ast(D)$. As $F^\ast$ is fully faithful, this map arises uniquely from a map $\psi_1 \colon P \to P'$ commuting with $g$ as in \eqref{eqn.typedmorphism}.

Using the observations that $\phi_1^{A,z}(x) = \psi_1 \circ x$ and $F^\ast$ preserves pullbacks, we can equivalently express $\phi^\#_x$ as a map of the form $\bar x^\ast F^\ast(\psi_1^\ast P'_\ast) \to \bar x^\ast F^\ast(P_\ast)$, natural in $\bar x = F^\ast(x) \colon a[A] \to F^\ast(P)$. By the canonical colimit decomposition of $F^\ast(P)$ in $a\set$ and preservation of colimits by pullbacks, such a $F/P$-indexed natural transformation is uniquely determined by a morphism $F^\ast(\psi_1^\ast P'_\ast) \to F^\ast(P_\ast)$ which commutes over $F^\ast(P)$ and, by previous assumption on $\phi^\#_x$, over $F^\ast(D')$ as well. As $F^\ast$ is fully faithful, this is equivalently a morphism $\psi_1^\ast P'_\ast \to P_\ast$ over $P$ in $\E$ which also commutes over $D'$.

In conclusion, we have shown that squares in $\ccatsharp$ filling in the boundary associated to that of \eqref{eqn.squareframe} from $\ppoly_\E$ correspond bijectively with squares of this form in $\ppoly_\E$, completing the proof that the desired double functor is locally fully faithful, and thereby a pseudo-double functor.
\end{proof}

\begin{example}
Let $\cat{E}\coloneqq\smcat$ as a 1-category, and let $a\coloneqq\mathbf{\Delta}\op$ be the simplicial indexing category, with $F\colon a\op\to\cat{E}$ the fully faithful and dense functor sending $N\mapsto\vec{N}$. The exponentiable maps $\pi\colon\cat{E}\to\cat{B}$ in $\smcat$ are \emph{Conduch\'e} fibrations. The functor $\poly_\smcat\to\ccatsharp$ sends $\pi$ to the bicomodule $\mathbf{\Delta}\op\bifrom[p]\mathbf{\Delta}\op$ where $e$ is carried by
\[
  p\coloneqq
	\sum_{N\in\ob(\mathbf{\Delta})}\sum_{x\colon\smcat(\vec{N},\cat{B})}
	\yon^{\sum_{N'\in\ob(\mathbf{\Delta})}\smcat(\vec{N'},x^*\pi)}
\]
Thus the simplicial set of $p$-positions is the nerve of $\cat{B}$ and for each $N$-simplex $x$ in it, the simplicial set of $p$-directions is the nerve of the fiber of $\pi$ over $x$.
\end{example}

\section{Structures in $\poly_\E$ and $\ccatsharp$}

Following \cite[Proposition 2.7.1]{spivak2021functorial}, the category of $(c,d)$-bicomodules has a monoidal structure $(c(1)\yon^{d(1)},\ot{c}{d})$ where the tensor product is given by
\[
p \ot{c}{d} q \coloneq \sum_{C \in c(1)} \sum_{(I,J) \in p_C(1) \times q_C(1)} \yon^{p[I] \times_{d(1)} q[J]},
\]
where the fiber product $\times_{d(1)}$ of directions is the product on $d$-copresheaves. The composition product $\tri$ has a right coclosure $\lens{-}{-}$ (\cref{coclosure}). 

Meanwhile, in \cite[Chapter 4]{shapiro2023structures}, a tensor product $\otimes$, closure for $\otimes$, right coclosure for $\tri$, and indexed left coclosure for $\tri$ are defined in $\poly_\E$, though the (co)closures require additional assumptions on the category $\E$. % , and these structures (when they exist) are all preserved from $\poly_\E$ to $(a,a)$-bicomodules. 
In particular, when $\E$ has finite limits the Dirichlet tensor product on $\poly_\E$ is defined as the categorical product of morphisms (though this is not a product in the category $\poly_\E$), and is shown to form a duoidal category with the composition product. The unit of both monoidal structures is $\yon$, the identity morphism on the terminal object.

%\begin{corollary}[Proven as \cref{Etensorproof,coclosureproof}]\label{preservesstructure}
%The monoidal functor of \cref{polyEembed} is lax monoidal with respect to $\otimes$ and $\ot{a}{a}$, and preserves $\tri$-coclosures up to coherent natural isomorphisms when they exist.
%\end{corollary}

%When the category $\E$ has all finite limits, the product in $\E$ applied to all diagram types in \cref{defppoly} endows $\ppoly_\E$ with a symmetric normal colax monoidal structure denoted $\otimes$. The colax compositors arise from the maps
%\[
%\Pi_{p \times p'}((Q \times Q') \times_{D \times D'} (P_\ast \times P'_\ast)) \cong \Pi_{p \times p'}((Q \times_D P_\ast) \times (Q' \times_{D'} P'_\ast)) \to \Pi_p(Q \times_D P_\ast) \times \Pi_{p'}(Q' \times_{D'} P'_\ast)
%\]
%induced by product projections. This specializes to the Dirichlet product on $\poly_\E$ \cite[Definition 4.1]{shapiro2023structures} (as $\poly_\E$ is equivalent to the category of typed polynomials in $\E$ from the terminal object to itself), which is duoidal with respect to the composition monoidal structure on $\poly_\E$.

\begin{corollary}\label{Etensorproof}
The monoidal functor of \cref{polyEembed} is lax monoidal with respect to $\otimes$ and $\ot{a}{a}$.
\end{corollary}

The failure of strong monoidality here arises from the fact that in the category $a\set$, the fibers of a product of morphisms are given not by products of fibers but by fiber products.

\begin{proof}%%%use composite of functors
After unwinding the definitions we can see that for polynomials $P_\ast \to P$ and $Q_\ast \to Q$ in $\E$,
\[
\sum_{A \in a(1)}\sum_{x \colon F(A) \to P \times Q} \yon^{F^\ast\left(x^\ast(P_\ast \times Q_\ast)\right)}
\]
\[
\cong \sum_{A \in a(1)}\sum_{\substack{x_1 \colon F(A) \to P \\ x_2 \colon F(A) \to Q}} \yon^{F^\ast\left(x_1^\ast(P_\ast) \times_{F(A)} x_2^\ast(Q_\ast)\right)}
\]
\[
\cong \sum_{A \in a(1)}\sum_{\substack{x_1 \colon F(A) \to P \\ x_2 \colon F(A) \to Q}} \yon^{F^\ast\left(x_1^\ast(P_\ast)\right) \times_{a[A]} F^\ast\left(x_2^\ast(Q_\ast)\right)}
\]
by the fact that the pullback of a morphism into a product is the pullback of the pullbacks of the component morphisms, and the functor $F^\ast \colon \E \to a\set$ preserves pullbacks. Here $a[A]$ denotes the corepresentable $a$-copresheaf, which agrees with $F^\ast F(A)$ as $F$ is fully faithful.

This $(a,a)$-bicomodule has a morphism from the corresponding tensor product of $(a,a)$-bicomodules
\[
\left(\sum_{A \in a(1)}\sum_{x_1 \colon F(A) \to P} \yon^{F^\ast(x_1^\ast P_\ast)}\right) \ot{a}{a} \left(\sum_{A \in a(1)}\sum_{x_2 \colon F(A) \to Q} \yon^{F^\ast(x_2^\ast Q_\ast)}\right)
\]
\[
= \sum_{A \in a(1)}\sum_{\substack{x_1 \colon F(A) \to P \\ x_2 \colon F(A) \to Q}} \yon^{F^\ast\left(x_1^\ast(P_\ast)\right) \times_{a(1)} F^\ast\left(x_2^\ast(Q_\ast)\right)}
\]
induced by the inclusion from a fiber product into a product of $a$-copresheaves on the directions.

The functor sends the unit $\yon \colon 1 = 1$ to the $(a,a)$-bicomodule
\[
\sum_{A \in a(1)} \sum_{F(A) \to 1} \yon^{F^\ast F(A)} \cong \sum_{A \in a(1)} \sum_{F(A) \to 1} \yon^{a[A]},
\]
which likewise admits a map from the unit $a(1)\yon^{a(1)}$ of the monoidal structure on $(a,a)$-bicomodules, induced by the unique map $a[A] \to a(1)$ to the terminal $a$-copresheaf on directions.

The unit and associativity equations are then straightforward to deduce from the universal property of products.
\end{proof}

%monoidality of the double functor? Overcategory construction is lax monoidal (inclusion of a pullback into a product) but monoidality of Cat# is complicated so maybe not worth it

%%motivate this more
%The following additional preservation result follows immediately from fully faithfulness and strong monoidality for $\tri$, as the universal property of the coclosure is then preserved and reflected.
%
%\begin{corollary}\label{coclosureproof}
%The monoidal functor of \cref{polyEembed} preserves the $\tri$-closures up to coherent natural isomorphisms whenever they exist.
%\end{corollary}

\chapter{Open dynamics and computational effects in $\ccatsharp$}

Algebraic effects and effect handlers are a popular way of working with side effects in functional programming languages and they have received much research and development in the last decade, both via new languages and integration into current functional languages like OCaml, Haskell, and Scala \cite{leijen_koka_2014, bauer_programming_2015, ocaml_ocaml_2022, king_7942_2022, odersky_capabilities_2022}.

The idea is that instead of directly implementing side effects, an effectful program should instead signal that a side effect should be performed, and another program should ``handle'' that signal, afterwards returning control flow to the original program along with the result of that effect. The advantage of this is that side effects can be handled in different ways. For instance, the side-effect of accessing a database could be implemented with an in-memory database, an on-file database, a dummy database, or a database distributed across the entire world. The application logic should be indifferent to this implementation.

We can model a program that uses effects as a polynomial coalgebra, i.e. a set of ``states'' $S$ along with a function $\vartheta \colon S \to p(S)$ for some polynomial $p$. The positions $I \in p(1)$ represent the different effects that can be ``thrown'', and then the directions $x \in p[I]$ represent the possible results of that effect returned to the program. Given a state $s \in S$, $\vartheta(s)$ represents running the program until the next effect is thrown, and then returning that effect $I$ along with a continuation function $p[I] \to S$ saying what the next state is given the result of the effect. A position $I \in p(1)$ with $p[I] = \emptyset$ signals termination of the program.

An effects handler is then something which ``migrates'' a $p$-coalgebra to a $q$-coalgebra. For instance, this could translate abstract database accesses into UNIX system calls to the network stack. 

\begin{example}
We can represent effect types in a language like Haskell using a GADT (generalized algebraic data type) with a single type parameter that looks something like the following code.
\begin{lstlisting}[language=HaskellMin]
data DBQuery a where
  Search :: String -> DBQuery [Id]
  Retrieve :: Id -> DBQuery Record
\end{lstlisting}
This represents an API with two operations. The first operation allows you to search based on a string and returns a list of ids that match the query. The second operation allows you to retrieve the full record for a given id.

Mathematically, this is a presentation of the polynomial functor
\[ p = \mathtt{String}\; y^{\mathbf{List}(\mathtt{Id})} + \mathtt{Id}\; y^{\mathtt{Record}}. \]
Then a coalgebra for \texttt{DBQuery} would be a type \texttt{s} along with a function of type

\begin{lstlisting}[language=HaskellMin]
s -> (exists a. (DBQuery a, a -> s))
\end{lstlisting}

	For \texttt{DBQuery}, the only options for \texttt{a} in the above are \texttt{[Id]} or \texttt{Record}; in general \texttt{a} ranges over the possible return types for an effect.
\end{example}

Mathematically, the abstract form of an effects handler can be modeled in the language of polynomials. As we show in \cref{sec.effects}, this allows for the construction of a pseudo-double category $\eff$ whose horizontal morphisms are effects handlers along with a pseudo-double functor $\eff \to \ccatsharp$. This nearly faithful mapping of effects handlers into $\ccatsharp$ is interesting both in its own right for exhibiting $\ccatsharp$ as a setting in which to work with effectful functional programs, and as a factor in a locally fully faithful pseudo-double functor $\org \to \ccatsharp$. Here, $\org$ is the pseudo-double category described in \cite{spivak2021learners,shapiro2022dynamic} whose horizontal morphisms are polynomial coalgebras, providing an elegant polynomial-based setting for modeling discrete open dynamical systems. This composite result shows that even for working with dynamics in $\org$ (which the authors have sometimes called ``the other'' pseudo-double category of interest in the theory of polynomial functors) it suffices to consider only $\ccatsharp$.

A key ingredient in these comparisons is the construction of the \emph{cofree comonoid} $\cofree p$ from a polynomial $p$, so we begin by providing the construction of $\cofree p$ and proving that it is indeed a cofree comonoid.

\section{Cofree comonoids}

Much like the construction of free monoids, which are constructed using colimits in a manner left adjoint to a forgetful functor, \emph{cofree comonoids} are dually constructed using limits in a manner right adjoint to a forgetful functor. %cite literature for cofree comonoid construction?

\begin{definition}
Given a polynomial $p$, we define polynomials $p^{(i)}$ for $i\in\nn$ by
\[
  p\coh{0}\coloneqq\yon
  \qqand
  p\coh{1+i}\coloneqq\yon\times\left(p\tri p\coh{i}\right)
\]
There is a projection map $\pi\coh{0}\colon p\coh{1}\to p\coh{0}$, and if $\pi\coh{i}\colon p\coh{1+i}\to p\coh{i}$ has been defined, then we can define $\pi\coh{1+i}\coloneqq \yon\times(p\tri\pi\coh{i})$. Now define the polynomial
\begin{equation}\label{eqn.construct_cofree}
\cofree p\coloneqq\lim\left(\cdots\To{\pi\coh{2}}p\coh{2}\To{\pi\coh{1}}p\coh{1}\To{\pi\coh{0}}p\coh{0}\right)
\end{equation}
and we note that this construction $p\mapsto \cofree p$ extends to a functor $\cofree - \colon \poly \to \poly$.
\end{definition}

Given this definition of $\cofree p$, in order to treat it as the cofree comonoid on $p$ it remains to show that it is in fact a comonoid, and that it has the desired universal property which can be expressed using an adjunction.

\begin{proposition}\label{prop.cofree}
The polynomial $\cofree p$ has the structure of a $\tri$-comonoid for each $p:\poly$,
\[
\cofree p\to\yon
\qqand
\cofree p\to\cofree p\tri\cofree p.
\qedhere
\]
\end{proposition}

\begin{proof}
The polynomial $\cofree p$ comes equipped with a counit $\epsilon\colon\cofree p\to\yon=p\coh{0}$ given by the projection. We next construct the comultiplication $\delta\colon\cofree p\to\cofree p\tri\cofree p$. Since $\tri$ commutes with connected limits, we have
\[
  \cofree p\tri\cofree p=
  \left(\lim_{i_1}p\coh{i_1}\right)\tri\left(\lim_{i_2}p\coh{i_2}\right)\cong
  \lim_{i_1,i_2}\left(p\coh{i_1}\tri p\coh{i_2}\right)
\]
To obtain the comultiplication $\lim_ip\coh{i}\to\lim_{i_1,i_2}(p\coh{i_1}\tri p\coh{i_2})$, it suffices to produce a natural choice of polynomial map $\varphi_{i_1,i_2}\colon p\coh{i_1+i_2}\to p\coh{i_1}\tri p\coh{i_2}$ for any $i_1,i_2:\nn$. When $i_1=0$ or $i_2=0$, we use the unit identity for $\tri$. By induction, assume given $\varphi_{i_1,1+i_2}$; we construct $\varphi_{1+i_1,1+i_2}$ as follows:
\begin{align}
\nonumber
  p\coh{1+i_1+1+i_2}&=
  \yon\times \left(p\tri p\coh{i_1+1+i_2}\right)\\&\to
\label{eqn.induction}
  \yon\times \left(p\tri p\coh{i_1}\tri p\coh{1+i_2}\right)\\&\to
\label{eqn.special}
  \left(\yon\times p\tri p\coh{i_1}\right)\tri p\coh{1+i_2}\\&=
\nonumber
  p\coh{1+i_1}\tri p\coh{1+i_2}
\end{align}
where \eqref{eqn.induction} is $\varphi_{i_1,1+i_2}$ and it remains to construct \eqref{eqn.special}. Recall that $-\tri q$ preserves products for any $q$, so constructing \eqref{eqn.special} is equivalent to constructing two maps
\[
\yon\times \left(p\tri p\coh{i_1}\tri p\coh{1+i_2}\right)\To{\phi\coh{i_1,i_2}} p\coh{1+i_2}
\qqand
\yon\times \left(p\tri p\coh{i_1}\tri p\coh{1+i_2}\right)\to p\tri p\coh{i_1}\tri p\coh{1+i_2}.
\]
For the latter we use the second projection. The former, $\phi\coh{i_1,i_2}\colon p\coh{1+i_1+1+i_2}\to p\coh{1+i_2}$, is the more interesting one; for it we also use projections $p\coh{i_1}\to p\coh{0}=\yon$ and $\pi\coh{i_2}\colon p\coh{i_2+1}\to p\coh{i_2}$ to obtain:
\[
\yon\times \left(p\tri p\coh{i_1}\tri p\coh{1+i_2}\right)\to
\yon\times \left(p\tri\yon\tri p\coh{i_2}\right)\cong p\coh{1+i_2}
\]
We leave the naturality of this to the reader.

It remains to check that $\epsilon$ and $\delta$ satisfy counitality and coassociativity. The base cases above imply counitality. Proving coassociativity amounts to proving that the following diagram commutes:
\[
\begin{tikzcd}[column sep=50pt]
	p\coh{1+i_1+1+i_2+1+i_3}\ar[r, "\phi\coh{i_1,i_2+1+i_3}"]\ar[d, "\phi\coh{i_1+1+i_2,i_3}"']&
	p\coh{1+i_2+1+i_3}\ar[d, "\phi\coh{i_2,i_3}"]\\
	p\coh{1+i_3}\ar[r,equal]&p\coh{1+i_3}
\end{tikzcd}
\]
This can be shown by induction on $i_3$.
\end{proof}

\begin{theorem}\label{thm.cofree_comonad_comonad}
There is an adjunction
\[
\adj{\catsharp}{U}{\cofree{-}}{\poly}
\]
where $U\colon\catsharp\to\poly$ is the forgetful functor $U(c,\epsilon,\delta)\coloneqq c$.
\end{theorem}

\begin{proof}
We will abuse notation and denote the comonoid $(c,\epsilon,\delta):\catsharp$ simply by its carrier $c$. We first provide the counit and unit of the desired adjunction. The counit
\[
\epsilon_p\colon\cofree{p}\to p
\]
is given by composing the projection map $\cofree{p}\to p\coh{1}$ from construction \eqref{eqn.construct_cofree} with the projection $p\coh{1}\cong\yon\times p\to p$. Since $\cofree{c}$ is defined as a limit, the unit
\[
\eta_c\colon c\coto\cofree{c}
\]
will be given by defining maps $\eta\coh{i}\colon c\to c\coh{i}$ commuting with the projections $\pi\coh{i}\colon c\coh{1+i}\to c\coh{i}$, for each $i:\nn$, and then showing that the resulting polynomial map $\eta_c$ is indeed a cofunctor. Noting that $c\coh{0}=\yon$, we define
\[
\eta\coh{0}\coloneqq\epsilon
\]
Given $\eta\coh{i}\colon c\to c\coh{i}$, we define $\eta\coh{1+i}$ as the composite
\[
c\To{(\epsilon,\delta)}\yon\times(c\tri c)\To{\yon\times(c\tri\eta\coh{i})}\yon\times\left(c\tri c\coh{i}\right)=c\coh{1+i}.
\]
Clearly, we have $\eta\coh{0}=\pi\coh{0}\circ\eta\coh{1}$. It is easy to check that if $\eta\coh{i}=\pi\coh{i}\circ\eta\coh{1+i}$ then $\eta\coh{1+i}=\pi\coh{1+i}\circ\eta\coh{2+i}$. Thus we have constructed a polynomial map $\eta\colon c\to \cofree{c}$. It clearly commutes with the counit, so it suffices to show that $\eta$ commutes with the comultiplication, which amounts to showing that the following diagram commutes
\[
\begin{tikzcd}
  c\ar[r, "\delta"]\ar[d, "{(\epsilon,\delta)}"']&[27pt]
  c\tri c\ar[r, "{(\epsilon,\delta)\tri(\epsilon,\delta)}"]&[15pt]
  \big(\yon\times(c\tri c)\big)\tri\big(\yon\times(c\tri c)\big)\ar[d, "{(\yon\times c\tri\eta\coh{i_1})\tri(\yon\times c\tri\eta\coh{i_2})}"]\\
  \yon\times(c\tri c)\ar[r, "{\yon\times c\tri\eta\coh{i_1+1+i_2}}"']&
  \yon\times\left(c\tri c\coh{i_1+1+i_2}\right)\ar[r, "\varphi_{1+i_1,1+i_2}"']&
  \left(\yon\times(c\tri c\coh{i_1})\right)\tri\left(\yon\times(c\tri c\coh{i_2})\right)
\end{tikzcd}
\]
for all $i_1,i_2:\nn$, where $\varphi_{1+i_1,1+i_2}$ is the map constructed in \cref{eqn.induction,eqn.special}. Commutativity follows from the counitality and coassociativity of the comonoid $c$.

The triangle identities are straightforward as well. Indeed, for any comonoid $c:\catsharp$, the composite $c\To{U\circ\eta_c} \cofree{c}\To{\epsilon_{Uc}} c$ is equal to the composite of $c\To{(\epsilon,c)}c\coh{1}=\yon\times c$, with the projection $c\coh{1}\to c$, the result of which is the identity. Finally, for any polynomial $p:\poly$, the composite $\cofree{p}\To{\eta_{\cofree{p}}}\cofree{\cofree{p}}\To{\cofree{\epsilon_p}}\cofree{p}$ is given by taking a limit of maps of the form
\[
	\cofree{p}\To{(\epsilon,\delta)}
	\yon\times\left(\cofree{p}\tri\cofree{p}\right)\To{\yon\times\left(\cofree{p}\tri\eta\coh{i}\right)}
	\yon\times\left(\cofree{p}\tri\cofree{p}\coh{i}\right)\To{\yon\times\left(\epsilon_p\tri\epsilon_p\coh{i}\right)}
	\yon\times\left(p\tri p\coh{i}\right)
\]
Each one is in fact the projection $\cofree{p}\to p\coh{i+1}$, so the resulting map is the identity on $\cofree{p}$, completing the proof.
\end{proof}

We note an interesting subcategory of $\ccatsharp$ given by applying the cofree construction to spans of polynomials. Noting that $\poly$ has all limits, let $\sspan(\poly)$ denote the double category of polynomials, morphisms, and spans $p\from s\to q$ between them, and let $\sspan^\Cat{c}(\poly)$ denote the subcategory with the same objects and vertical morphisms, but for which a horizontal morphism is a span
\[
p\From{\cart}s\to q
\]
whose left leg is cartesian. These compose because the pullback of a cartesian map is cartesian.%
\footnote{We observe in passing that a monad in $\sspan^\Cat{c}(\poly)$ can be identified with a category equipped with a presheaf.}

\begin{proposition}
The functor $\cofree-\colon\poly\to\catsharp$ extends to a strong double functor
\[
	\cofree-\colon\sspan^\Cat{c}(\poly)\to\ccatsharp.
\]
In particular, it sends a span $p\From{\cart} s\to q$ to a bicomodule of the form
\[
\cofree p\bifrom[\cofree s]\cofree q
\]
and both identity and composition are preserved up to isomorphism.
\end{proposition}
\begin{proof}
In fact, there is a colax double functor $\sspan(\poly)\to\ccatsharp$, though we will not prove it here because we find preservation of composition more interesting. Given a span $p\From{\varphi} s\To{\psi} q$, the corresponding bicomodule has structure morphisms given by composites as shown:
\[
	\cofree p\tri\cofree s\From{\varphi\tri\cofree s}
	\cofree s\tri \cofree s\From{\delta}
	\cofree s\To{\delta}
	\cofree s\tri \cofree s\To{\cofree s\tri\psi}
	\cofree s\tri\cofree q
\]
It is clear that this mapping preserves identities. Suppose given composable spans
\[
p\From{\cart} s\to q\From{\cart} t\to r
\]
for which both left legs are cartesian, and let $p\from (s\times_q t)\to r$ be the composite. We will be done if we can show that there is an isomorphism of bicomodules
\[
\begin{tikzcd}[row sep=5pt]
	&\cofree q\ar[dr, bimr-biml, bend left=20pt, "\cofree t"]\\
	\cofree p\ar[ur, bimr-biml, bend left=20pt, "\cofree s"]\ar[rr, bimr-biml, bend right=10pt, "\cofree{s\times_qt}"']&\ar[u, equal]&
	\cofree r
\end{tikzcd}
\]
In other words, we need to show that the following is an equalizer diagram:
\begin{equation}\label{eqn.kappa}
\begin{tikzcd}[row sep=0pt]
	&&&
	\cofree s\tri\cofree s\tri\cofree t\ar[rd]\\
	\cofree{s\times_qt}\ar[r, "\delta"]&
	\cofree{s\times_qt}\tri\cofree{s\times_qt}\ar[r]&
	\cofree s\tri\cofree t\ar[ur]\ar[dr]&&
	\cofree s\tri\cofree q\tri \cofree t\\&&&
	\cofree s\tri\cofree t\tri\cofree t\ar[ru]
\end{tikzcd}
\end{equation}
We first check that the two composites are equal. To see this, we embed the above diagram in a larger one
\[
\begin{tikzcd}[row sep=5pt]
	&&
	\cofree{s\times_qt}\tri\cofree{s\times_qt}\tri\cofree{s\times_qt}\ar[r]&
	\cofree s\tri\cofree s\tri\cofree t\ar[rd]\\
	\cofree{s\times_qt}\ar[r, "\delta"]&
	\cofree{s\times_qt}\tri\cofree{s\times_qt}\ar[r]\ar[ur]&
	\cofree s\tri\cofree t\ar[ur]\ar[dr]&&
	\cofree s\tri\cofree q\tri \cofree t\\&&&
	\cofree s\tri\cofree t\tri\cofree t\ar[ru]\ar[from=uul, crossing over]
\end{tikzcd}
\]
Because all three new squares commute, the required composites are indeed equal. 

Now suppose given a polynomial $x$ and a map $\alpha\colon x\to\cofree s\tri\cofree t$ making the composites commute; we need to provide a unique map to $\cofree{s\times_qt}$. Note that $a\tri-\tri b$ preserves all connected limits---in particular pullbacks---and that $\cofree-$ also preserves pullbacks. Thus we have an induced map $x\to\cofree s\tri\cofree{s\times_qt}\tri\cofree t$, which we can compose with the counits on $\cofree s$ and $\cofree t$ to obtain the desired map $x\to\cofree{s\times_qt}$. It is easy to check that composing it with $\cofree{s\times_qt}\to\cofree s\tri \cofree t$ returns $\alpha$ as necessary.

It remains to show that the map we obtained is unique, and for that it suffices to show that the map $\kappa\colon\cofree{s\times_qt}\to\cofree s\tri\cofree t$, as shown in \eqref{eqn.kappa}, is monic, i.e.\ that $\kappa_1$ is injective and that for each $I:\cofree{s\times_qt}(1)$, the function $\kappa^\sharp_I\colon(\cofree s\tri\cofree t)[\kappa_1I]\to\cofree{s\times_qt}[I]$ is surjective. This is where we bring in the fact that $t\to q$ is Cartesian; it implies that $s\times_qt\to s$ is also cartesian. Thus we can identify $I$ with a tuple $(S,T,f)$ where $S:s(1)$, $T:t(1)$, and $f\colon s[S]\to\cofree{s\times_qt}(1)$, since $(s\times_qt)[(S,T)]\cong s[S]$. If $\kappa_1(S,T,f)=\kappa_1(S',T',f')$, one check immediately that $S=S'$ and $T=T'$, and by induction that $f=f'$. And $\kappa^\sharp_I$ is given by first projection, which is surjective since $\cofree t$ at least has nonempty direction sets (it at least has identity morphisms). This completes the proof.
\end{proof}

The following is immediate, since maps $A\to B$ in $\smset$ are cartesian as maps in $\poly$.

\begin{corollary}
The double category $\sspan(\smset)$ embeds into $\ccatsharp$. 
\end{corollary}

Myers \cite[Section 3.5]{myers2023categorical} defines the double category $\mathbb{A}\Cat{rena}_{\smset/-}$ of dependent arenas, whose objects are polynomials, whose vertical maps are polynomial maps (there called \emph{lenses}), and whose horizontal maps are \emph{charts}. A chart between polynomials $p$ and $q$ is just a bundle map between the associated bundles, from $p_*(1)\to p(1)$ to $q_*(1)\to q(1)$. A 2-cell is just a map of spans. Thus the following corollary is again immediate.

\begin{corollary}
There is a double functor $\mathbb{A}\Cat{rena}_{\smset/-}\to\ccatsharp$.
\end{corollary}

\section{Effects handlers}\label{sec.effects}

\begin{definition}\label{def.eff}
For polynomial comonoids $c,d$, a \emph{$(c,d)$-effects handler} is a pair $(s,\varphi)$ where the \emph{carrier} $s$ is a polynomial and $\varphi$ is a morphism $c \tri s \From{\varphi} s \tri d$ which commutes with counits and comultiplications in the sense of \cref{eqn.eff}. We say it is \emph{linear} if $s=S\yon$ for some $S:\smset$.
\begin{equation}\label{eqn.eff}
\begin{tikzcd}
	c\tri s\ar[d, "\epsilon_c\tri s"']&
	s\tri d\ar[d, "s\tri\epsilon_d"]\ar[l, "\varphi"']\\
	s\ar[r, equal]&
	s
\end{tikzcd}
\hspace{.6in}
\begin{tikzcd}
	c\tri s\ar[d, "\delta_c\tri s"']&&
	s\tri d\ar[d, "s\tri\delta_d"]\ar[ll, "\varphi"']\\
	c\tri c\tri s&
	c\tri s\tri d\ar[l, "c\tri\varphi"]&
	s\tri d\tri d\ar[l, "\varphi\tri d"]
\end{tikzcd}
\end{equation}

For polynomials $p,q$, a \emph{$(p,q)$-effects handler} is a polynomial $s$ equipped with a morphism $p \tri s \From{\varphi} s \tri q$. We refer to these as \emph{elementary effects handlers}. 
\end{definition}

\begin{example}
	Let $p = y$, $q = y^\nn$, and $s = \nn y$. We interpret these as follows. The polynomial $q$ describes ``a single effect, with return type $\nn$''. It's a button that you can push, and you get a natural number when you push the button. The polynomial $p$ we interpret as a single effect with unit return type. You can push the button, but you always get the same result. Then $s$ represents a state machine with $\nn$ states. In each state, we are handling effects from precisely one $q$-coalgebra.

	We are going to describe an effects handler that implements the following (very dumb) game. There are two players, $q$-Bob and $p$-Bob. $q$-Bob asks for a natural number. $p$-Bob sees $q$-Bob's request and approves it. Then $q$-Bob gets a natural number. It is always one more than the last natural number he got.

	This is implemented by an elementary effects handler $p \tri s \From{\varphi} s \tri q$ in the following way. A position of $s \tri q$ is a pair $(n \in \nn, f \colon 1 \to 1)$. A position of $p \tri s$ is a pair $(x \in 1, u \colon 1 \to \nn)$. Both of these are just isomorphic to $\nn$, so we can say that the action of $\varphi$ on positions is just $\varphi(n) = n + 1$. The direction set at any position of $s \tri q$ is the natural numbers, and the direction set at any position of $p \tri s$ is the singleton. We then define the backwards direction by $\varphi_n^\sharp(\ast) = n$.
\end{example}

\begin{example}
	Let $\mathbf{Prog}$ be the set of programs in a given programming language, and let $p = \mathbf{Prog}\; y^2 + y$. Let $q = \{\blacksquare, \square\}\;y^{\{\leftarrow, \rightarrow, \blacksquare, \square\}}$. Then we can interpret a Turing machine with access to a Halting oracle as an elementary $(p,q)$-effects handler. At each step, the state machine controlling the Turing machine gets to read the current tape position, which is either $\blacksquare$ or $\square$. Then the Turing machine can either output a program and get back a yes-no answer, or just output a request to keep going. Finally, the Turing machine returns a new instruction in $\{\leftarrow, \rightarrow, \blacksquare, \square\}$ to the tape, which says to move the head left or right, or to write $\blacksquare$ or $\square$ to the current position.

	Of course, there is nothing here which says that the Turing machine has to be hooked up to a correct Halting oracle; it might be hooked up to something which is just returning yes or no based on a pseudo-random number generator. But that's the point: the description of the Turing machine itself should treat the oracle as ``external.''
\end{example}

\begin{definition}%check if equipment
The pseudo-double category $\eff$ has polynomial comonoids as objects, comonoid homomorphisms as vertical morphisms, and $(c,d)$-effects handlers as horizontal morphisms from $d$ to $c$. Given comonoid homomorphisms $c \to c'$ and $d \to d'$ we can define a square from a $(c,d)$-effects handler $s$ to a $(c',d')$-effects handler $s'$ as a morphism of polynomials $s \to s'$ which commutes with the effects handler structure maps as in \eqref{eqn.effectsquare}.
\begin{equation}\label{eqn.effectsquare}
\begin{tikzcd}
	c \tri s \dar & s \tri d \dar \lar \\
	c' \tri s' & s' \tri d' \lar
\end{tikzcd}
\end{equation}

For a comonoid $c$, its identity $(c,c)$-effects handler is given by the polynomial $c$ and the identity morphism $c \tri c \from c \tri c$, and any comonoid homomorphism $c \to c'$ induces a horizontal identity square between the identity $(c,c)$- and $(c',c')$-effects handlers. For a $(c,d)$-effects handler $s$ and a $(d,e)$-effects handler $t$, we get a $(c,e)$-effects handler $s \tri t$ given by the composite $c \tri s \tri t \from s \tri d \tri t \from s \tri t \tri e$. $\tri$ similarly defines a horizontal composition of squares, while the unitors and associators for horizontal composition are given by those of $\tri$.

There is similarly a pseudo-double category $\effel$ of elementary effects handlers without the comonoid structure on objects or morphisms.
\end{definition}

As one might expect, elementary effects handlers can be regarded as effects handlers. Given an elementary effects handler $p \tri s \from s \tri q$, we will construct an effects handler with the same carrier of the form $\cofree p \tri s \from s \tri \cofree q$ as part of a pseudo-double functor $\effel \to \eff$.

\begin{lemma}\label{lemma.eff_handler_to_cofree}
For a polynomial $p$ and a polynomial comonoid $c$, we can identify $(\cofree p,c)$-effects handlers with elementary $(p,c)$-effects handlers.
\end{lemma}

\begin{proof}
Given a $(\cofree p,c)$-effects handler $\cofree p\tri s\From{\psi} s\tri c$, one composes with the projection $\cofree p\to p$ to obtain an elementary $(p,c)$-effects handler. 

Going the other way, suppose we are given an elementary effects handler $p\tri s\From{\varphi}s\tri c$. This can be identified with a map
\[
\lens{s\tri c}{s}\to p.
\]
Since the lefthand side is a comonoid by \cref{lemma.selection}, the universal property of $\cofree{-}$ implies that the map factors uniquely through a cofunctor
\[
\lens{s\tri c}{s}\coto \cofree p\to p.
\]
We unfold the first factor as the associated effects handler, $\cofree p\tri s\from s\tri c$.

\end{proof}

\begin{corollary}\label{prop.elem_as_eff}
There is a double functor $\effel \to \eff$ extending the functor $\cofree -$ on the vertical category.
\end{corollary}

\begin{proof}
We first show that an elementary $(p,q)$-effects handler gives rise to a $(\cofree p, \cofree q)$-effects handler. Given a $(p,q)$-effects handler $p\tri s\From{\varphi}s\tri q$, we compose with the projection $s\tri q\From{\varphi\circ(s\tri\epsilon)} s\tri\cofree q$ to obtain an elementary $(p,\cofree q)$-effects handler. Then by \cref{lemma.eff_handler_to_cofree}, we can identify it with a $(\cofree p,\cofree q)$-effects handler, which we denote $\cofree p\tri s\From{\wt{\varphi}}s\tri\cofree q$. By naturality of the projection morphism, this construction extends to squares. It remains to show that it preserves identities and composites.

The identity elementary effects handler on $p$ is ``the identity'', $\yon\tri p\cong p\tri\yon$, and it is sent by the above construction to ``the identity'' $\yon\tri\cofree p\cong\cofree p\tri\yon$. For composition, suppose we are given $p\tri s\From{\varphi} s\tri q$ and $q\tri t\From{\psi} t\tri r$, and consider the following diagram:
\[
\begin{tikzcd}[row sep=48pt, column sep=60pt]
  \lens{s\tri\lens{t\tri \cofree r}{t}}{s}
  	\ar[r, "\lens{s\tri\wt{\psi}}{s}"]
		\ar[d, "\lens{s\tri\lens{t\tri \epsilon}{t}}{s}"']
		\ar[dr, pos=.4, "\lens{s\tri(\psi\circ(t\tri\epsilon))}{s}"]&[20pt]
	\lens{s\tri\cofree q}{s}\ar[r, "\wt{\varphi}"]\ar[d, "\lens{s\tri\epsilon}{s}"]&
	\cofree p\ar[d, "\epsilon"]\\
	\lens{s\tri\lens{t\tri r}{t}}{s}\ar[r, "\lens{s\tri\psi}{s}"']&
	\lens{s\tri q}{s}\ar[r, "\varphi"']&
	p
\end{tikzcd}
\]
The righthand square and the top-left triangle commute by construction of $\wt{\varphi}$ and $\wt{\psi}$, and the bottom-left triangle commutes by definition. The bottom composite unfolds to that of $\varphi\circ\psi$ as elementary effects handlers, whereas the top composite unfolds to that of $\wt{\varphi}\circ\wt{\psi}$ as effects handlers, and the vertical arrows construct the mapping between them and show that it preserves composition, completing the proof.
\end{proof}

\begin{lemma}\label{lemma.tri_d_compose_over_d}
For any bicomodules of the form $e\bifrom[q\tri d]d\bifrom[p]c$, where the right comodule structure $q\tri d\to q\tri d\tri d$ is given by $q\tri \delta$, the composite has the form
\[
\begin{tikzcd}[row sep=5pt]
	&d\ar[dr, bimr-biml, bend left=20pt, "p"]\\
	e\ar[ur, bimr-biml, bend left=20pt, "q\tri d"]\ar[rr, bimr-biml, bend right=10pt, "q\tri p"', "" name=comp]&\ar[u, equal]&
	c
\end{tikzcd}
\] 
 with structure maps given by the following composites:
\[
e\tri q\tri p\From{e\tri q\tri\epsilon\tri p} 
e\tri q\tri d\tri p\From {\lambda\tri p}
q\tri d\tri p \From{q\tri\lambda}
q\tri p\To{q\tri\rho}
q\tri p\tri c
\]
\end{lemma}
\begin{proof}
Composite bicomodules are given by an equalizer; in our case, we need to show that
 $q\tri p\to q\tri d\tri p\tto q\tri d\tri d\tri p$ is an equalizer. But equalizers are preserved by $\tri$ in either variable, so it suffices to show that the following is an equalizer:
\[
\begin{tikzcd}
	p\ar[r, "\lambda"]&
	d\tri p\ar[r, shift left=5pt, "d\tri\lambda"]\ar[r, shift right=5pt, "\delta\tri p"']&
	d\tri d\tri p
\end{tikzcd}
\]
By definition of left comodule the diagram commutes $\lambda\then(d\tri\lambda)=\lambda\then(\delta\tri p)$. Given a polynomial $x$ and a map $\varphi\colon x\to d\tri p$ such that $\varphi\then(d\tri\lambda)=\varphi\then(\delta\tri p)$, we obtain a map $x\To{\varphi} d\tri p\To{\epsilon\tri p} p$, and it is an easy calculation to show that postcomposing it with $\lambda$ returns $\varphi$, completing the proof.
\end{proof}

We now move on to the relationship between effects handlers and bicomodules: any $(c,d)$-effects handler $(s,\varphi)$ induces a $(c,d)$-bicomodule $c\bifrom[s\tri d]d$ with left and right structure maps $(\varphi\tri d)\circ(s\tri\delta_d)$ and $s\tri\delta_d$, as in \eqref{eqn.bicomod_for_eff}
\begin{equation}\label{eqn.bicomod_for_eff}
c \tri s \tri d \From{\varphi\tri d} s \tri d \tri d \From{s\tri\delta_d} s \tri d \To{s\tri\delta_d} s \tri d \tri d.
\end{equation}

\begin{theorem}\label{thm.eff}
The pseudo-double category of effects handlers admits a pseudo-double functor $\eff\to\ccatsharp$ which is the identity on objects and vertical morphisms and which is faithful on the category of horizontal morphisms between nonempty categories and squares between them. Moreover, every $(c,0)$- and $(c,\yon)$-bicomodule is in the essential image.
\end{theorem}

\begin{proof}
We first show that for any $(c,d)$-effects handler $(s,\varphi)$, the structure maps from \eqref{eqn.bicomod_for_eff} do in fact form a bicomodule $c\bifrom[s\tri d]d$. It is easy to check that $s\tri d\To{s\tri\delta}s\tri d\tri d$ is a right comodule. To check that the left structure maps commutes with counit, we have the following
\[
\begin{tikzcd}
	s\tri d\ar[r, "s\tri\delta"]\ar[dr, equal]&
	s\tri d\tri d\ar[r, "\varphi\tri d"]\ar[d, "s\tri\epsilon\tri d"]&
	c\tri s\tri d\ar[d, "\epsilon\tri s\tri d"]\\&
	s\tri d\ar[r, equal]&
	s\tri d
\end{tikzcd}
\]
where the triangle commutes because $d$ is a comonad, and the square commutes by \eqref{eqn.eff}. Checking that the left structure commutes with comultiplication is similar, and the compatibility between left and right structures is even easier.

A square in the double category $\eff$ is a map $\gamma\colon s\to s'$ and a commuting square
\[
\begin{tikzcd}
	c\tri s\ar[d, "\alpha\tri\gamma"']& 
	s\tri d\ar[d, "\gamma\tri\beta"]\ar[l, "\varphi"']\\
	c'\tri s'&
	s'\tri d'\ar[l, "\varphi'"]
\end{tikzcd}
\]
for cofunctors (comonoid homomorphisms) $\alpha\colon c\to c'$ and $\beta\colon d\to d'$. This gives rise to a square in $\ccatsharp$:
\[
\begin{tikzcd}
	c\ar[d, "\alpha"']\ar[r, bimr-biml, "s\tri d", ""' name=sd]&
	d\ar[d, "\beta"]\\
	c'\ar[r, bimr-biml, "s'\tri d'"', "" name=s'd']&
	d'\ar[from=sd, to=s'd', shorten=2mm, Rightarrow]
\end{tikzcd}
\]
Indeed, squares of this form are in bijection with $(c',d')$-bicomodule maps $s\tri d\to s'\tri d'$, and we obtain one from $\gamma$ as follows:
\[
\begin{tikzcd}
	c\tri s\tri d\ar[d]&
	s\tri d\tri d\ar[d]\ar[l]&
	s\tri d\ar[d]\ar[l]\ar[r]&
	s\tri d\tri d\ar[d]\\
	c'\tri s'\tri d'&
	s'\tri d'\tri d'\ar[l]&
	s'\tri d'\ar[l]\ar[r]&
	s'\tri d'\tri d'
\end{tikzcd}
\]
To see that this map is faithful for $d\neq 0$, suppose given maps $\gamma_1,\gamma_2\colon s\to s'$ which induce the same map $(\gamma_1\tri\beta)=(\gamma_2\tri\beta)\colon s\tri d\to s'\tri d'$. Then both squares below commute
\[
\begin{tikzcd}
	s\tri d\ar[r, "s\tri\epsilon"]\ar[d]&
	s\ar[d, shift right, "\gamma_1"']\ar[d, shift left, "\gamma_2"]\\
	s'\tri d'\ar[r, "s'\tri\epsilon"']&
	s'
\end{tikzcd}
\]
so it suffices to show that $s\tri\epsilon$ is an epimorphism. The operation $s\tri-$ preserves epimorphisms, and $p\to\yon$ is an epimorphism for any polynomial $p\neq 0$.

It follows from \cref{lemma.tri_d_compose_over_d} that horizontal composites and identities are preserved by our functor, e.g.\ we have natural isomorphisms
\[
\begin{tikzcd}[row sep=5pt]
	&d\ar[dr, bimr-biml, bend left=20pt, "t\tri e"]\\
	e\ar[ur, bimr-biml, bend left=20pt, "s\tri d"]\ar[rr, bimr-biml, bend right=10pt, "s\tri t\tri e"', "" name=comp]&\ar[u, equal]&
	c
\end{tikzcd}
\]

Finally, every bicomodule $c\bifrom[S]0$ gives rise to an effects handler, $c\tri S\from S=S\tri 0$. Similarly, every bicomodule $c\bifrom[s]\yon$ gives rise to a an effects handler $c\tri s\from s=s\tri\yon$. In both cases, the required commutativity \eqref{eqn.eff} follows from that of the bicomodules.
\end{proof}

\section{Polynomial coalgebras and the double category $\org$}

For a polynomial $p$, a $p$-coalgebra is a set $S$ of \emph{states} equipped with a function $f \colon S \to p \tri S$, which encodes an \emph{action} function $f_0 \colon S \to p(1)$ labeling each state with a position of $p$ and an \emph{update} function $f_s \colon p[f_0(s)] \to S$ indicating how each direction in $p[f_0(s)]$ transitions from $s \in S$ to a potentially new state in $S$. This can be regarded as a generalization of finite automata, where the polynomial $p$ encodes the set of labels a state can have and a set of outgoing transitions which depends on the label. 

\begin{definition}\label{pqcoalg}
The closure $[-,-]$ of the monoidal structure $(\yon,\otimes)$ on $\poly$ is given by
\[
[q,p] \coloneq \sum_{\phi \colon q \to p} \yon^{\sum\limits_{I \in q(1)} p[\phi_1I]}
\]
for polynomials $p,q$. A $[q,p]$-coalgebra is a set $S$ equipped with a function $S \to [q,p](S)$.
\end{definition}

In \cite{spivak2021learners,shapiro2022dynamic}, the authors describe a double category $\org$ whose vertical category is that of polynomials, whose horizontal morphisms from $q$ to $p$ are the $[q,p]$-coalgebras $S\to [q,p]\tri S$, and whose squares are maps $S \to S'$ satisfying a certain commutativity condition \cite[Section 2.4]{shapiro2022dynamic}. Monoidal categories and operads enriched in $\org$ can be used to model the process of training a deep learning system and running a prediction market.

It turns out that the category of $[q,p]$-coalgebras is equivalent to that of elementary $(p,q)$-effects handlers of the form $(S\yon,\varphi)$ for some $S:\smset$, i.e.\ whose carrier is linear,%
\footnote{Note that this is not a contravariant assignment, as a $(p,q)$-effects handler is regarded as a morphism from $q$ to $p$ in $\eff$, a convention inherited from $\ccatsharp$.} an assignment which furthermore extends to the entire structure of $\org$.

\begin{lemma}\label{lemma.dir_tri_lin_rep}
For any sets $S,T:\smset$ and polynomial $p:\poly$, the maps
\[
	S\yon\otimes p\To\cong S\yon\tri p
  \qqand
	p\otimes\yon^T\To\cong p\tri\yon^T
\]
induced by the duoidal structure \eqref{eqn.duoidal} are isomorphisms.
\end{lemma}

\begin{proof}
One checks that both maps are bijective on positions and directions.
\end{proof}

\begin{lemma}\label{lemma.const_linear}
For any sets $S,T:\smset$ and polynomial $p$ there is a natural bijection between hom-sets
\[
\poly(S\yon,p\tri T\yon)
\cong
\smset(S,p\tri T)
.
\]
\end{lemma}

\begin{proof}
The functor $(S\mapsto S\yon):\smset\to\poly$ is left adjoint to $(-\tri  1) \colon\poly\to\smset$.
\end{proof}

\begin{lemma}\label{lemma.lin_rep_adj}
For a set $S:\smset$ and polynomials $p,q:\poly$, there is a natural bijection between hom-sets
\[
\poly(p,q\tri S\yon)
\cong
\poly(p\tri\yon^S,q)
.
\]
\end{lemma}

\begin{proof}
For any $S$, the polynomial functor $S\yon$ is left adjoint to $\yon^S$, i.e.\ there is a unit $\yon\to\yon^S\tri S\yon$ and a counit $S\yon\tri\yon^S\to\yon$ satisfying the triangle equations. Given a map $p\to q\tri S\yon$ one applies $(-\tri\yon^S)$ to both sides and composes with the counit to obtain a map $p\tri\yon^S\to q$, and given a map of the latter form, one applies $(-\tri S\yon)$ to both sides and precomposes with the unit to obtain a map $p\to q\tri S\yon$. The round-trips are identities by the triangle equations.
\end{proof}

\begin{theorem}\label{thm.org}
There is a pseudo-double functor $\org\to\effel$ which is the identity on objects and vertical morphisms and fully faithful on the category of horizontal morphisms and squares, with essential image given by the linear elementary effects handlers.
\end{theorem}

\begin{proof}
The vertical categories of both $\org$ and $\effel$ are defined to be $\poly$. A horizontal morphism in $\org$ from $q$ to $p$ is a $[q,p]$-coalgebra; we want to show that these can be identified with linear elementary $(p,q)$-effects handlers. Define $\effellin(p,q)_S$ to be the category of linear elementary effects handlers with carrier $S\yon$ and define $x\coalg_S$ to be the category of $x$-coalgebras with carrier $S$. By \cref{lemma.dir_tri_lin_rep,lemma.const_linear,lemma.lin_rep_adj} and the adjunction $(-\otimes q)\dashv[q,-]$, we have the following isomorphisms, natural in $S$:
\begin{align*}
	[q,p]\coalg_S&\cong
	\poly(S,[q,p]\tri S)\\&\cong
	\poly(S\yon,[q,p]\tri S\yon)\\&\cong
	\poly(S\yon\tri\yon^S,[q,p])\\&\cong
	\poly(S\yon\otimes\yon^S,[q,p])\\&\cong
	\poly(S\yon\otimes q\otimes\yon^S,p)\\&\cong
	\poly((S\yon\otimes q)\tri\yon^S,p)\\&\cong
	\poly(S\yon\tri q,p\tri\yon^S)\\&\cong
	\effellin(p,q)_S
\end{align*}
It is straightforward to check that horizontal composition is preserved.

Squares in $\org$ of the form
\[
\begin{tikzcd}
	q\ar[r, slash, "S"]\ar[d, "\psi"']&
	p\ar[d, "\varphi"]\\
	q'\ar[r, slash, "S'"']&
	p'
\end{tikzcd}
\]
consist of maps $f\colon S\to S'$ such that the following diagram commutes:
\[
\begin{tikzcd}[column sep=45pt]
	S\ar[r, "\vartheta"]\ar[d, "f"']&
	{[q,p]\tri S}\ar[r, "{[q,\varphi]\tri S}"]&
	{[q,p']\tri S}\ar[d, "{[q,p']\tri f}"]\\
	S'\ar[r, "\vartheta'"']&
	{[q',p']\tri S'}\ar[r, "{[\psi,p']\tri S'}"']&
	{[q,p']\tri S'}	
\end{tikzcd}
\]
This condition is equivalent to that for squares in $\effel$, which demand that the following diagram commutes:
\[
\begin{tikzcd}
  S\yon\tri q\ar[r, "\vartheta"]\ar[d, "f\tri\psi"']&
  p\tri S\yon\ar[d, "f\tri \varphi"]\\
  S'\yon\tri q'\ar[r, "\vartheta'"']&
  p'\tri S'\yon
\end{tikzcd}
\]
Thus the squares agree, as do compositions of squares, completing the proof.
\end{proof}

We have now defined a string of locally fully faithful pseudo-double functors
\[
\org \to \effel \to \eff \to \ccatsharp
\]
which acts by $\cofree{} \colon \poly \to \poly$ on the vertical category and sends a coalgebra $S \to [q,p](S)$ to the $(\cofree p,\cofree q)$-effects handler $\cofree p \tri S\yon \from S\yon \tri \cofree q$ sending a state $s_0 \in S$ and a $q$-behavior tree $T$ to the behavior tree of $p$ obtained by running the coalgebra on $s_0$ and each state reached by the paths through $T$, labeled by the states reached along the way. 

This shows that $\ccatsharp$ is capable of modeling yet another of the major applications of polynomial functors; while it has until now been used primarily in the realm of categorical database theory, this shows that it also encodes the polynomial coalgebra formulation of discrete open dynamical systems.

\printbibliography

\end{document}